\documentclass[11pt]{amsart}

\newtheorem{theorem}{Theorem}[section]

\newtheorem{corollary}[theorem]{Corollary}

\newtheorem{definition}[theorem]{Definition}
\newtheorem{lemma}[theorem]{Lemma}

\newtheorem{proposition}[theorem]{Proposition}

\theoremstyle{definition}
\newtheorem{remark}[theorem]{Remark}

\usepackage[colorlinks, bookmarks=true]{hyperref}
\usepackage{color,graphicx,shortvrb}
\usepackage{enumerate}
\usepackage{amsmath}
\usepackage{amssymb}

\input xy
\xyoption{all}

\def\RR{{\mathbb{R}}}
\def\NN{{\mathbb{N}}}
\def\11{\textbf{$1$}}

\newcommand{\supp}{\mathrm{supp}}

\newcommand{\fatou}{{\mathfrak{f}}}

\newcommand{\one}{\mathbf{1}}

\newcommand{\vr}{\varepsilon}
\newcommand{\ave}{{\mathrm{Ave}}}

\newcommand{\BP}{\mathrm{BP}}
\newcommand{\ball}{\mathbf{B}}

\newcommand{\triple}[1]{| \! | \! | #1 | \! | \! |}

\begin{document}

\numberwithin{equation}{section}

\title[Almost band preservers]{Almost band preservers}

 \author[T. Oikhberg]{Timur Oikhberg}
 \address{Dept.~of Mathematics, University of Illinois Urbana IL 61801, USA}
 \email{oikhberg@illinois.edu}

 \author[P. Tradacete]{Pedro Tradacete}
\address{Mathematics Department, Universidad Carlos III de Madrid, E-28911 Legan\'es, Madrid, Spain.}
\email{ptradace@math.uc3m.es}

\thanks{T.O. partially supported by Simons Foundation travel award 210060.
P.T. partially supported by the Spanish Government grants MTM2013-40985, MTM2012-31286, and Grupo UCM 910346. }



\maketitle

\begin{abstract}
We study the stability of band preserving operators on Banach lattices.
To this end the notion of $\varepsilon$-band preserving mapping is introduced.
It is shown that, under quite general assumptions, a $\varepsilon$-band
preserving operator is in fact a small perturbation of a band preserving one.
However, a counterexample can be produced in some circumstances.
Some results on automatic continuity of $\varepsilon$-band preserving maps are also obtained.
\end{abstract}

\maketitle

\thispagestyle{empty}

\tableofcontents

\section{Introduction}

This paper is devoted to the stability of band preserving operators on a Banach lattice. Recall that a band $Y$ in a Banach lattice $X$ is an ideal (i.e. a subspace $Y$ such that if $y\in Y$ and $|x|\leq|y|$, then $x\in Y$) which is also closed under arbitrary suprema, i.e. for every collection $(y_\alpha)_{\alpha\in A}$ in $Y$ such that $\bigvee_{\alpha\in A} y_\alpha$ exists in $X$, this element must belong to $Y$. For instance, it is easy to see that on the spaces $L_p(\Omega,\Sigma,\mu)$, every band corresponds to the set of elements supported on some $A\in\Sigma$.

A linear operator on a Banach lattice $T:X\rightarrow X$ is \emph{band-preserving} (\emph{BP} for short) if $T(Y) \subset Y$ for any band $Y \subset X$. The study of band-preserving operators can be traced back to the work of H. Nakano \cite{Nakano}.
A complete account on band-preserving operators can be found in \cite[Section 3.1]{M-N}
and in \cite[Sections 2.3 and 4.4]{A-B}, see also the survey paper \cite{Hui}.
Let us recall a useful characterization of BP operators on a Banach lattice due to Y. Abramovich, A. Veksler and A. Koldunov \cite{AVK}: Given a Banach lattice $X$, and an operator $T\in B(X)$, the following are equivalent
\begin{enumerate}
\item $T$ is band preserving.
\item $T$ is an orthomorphism, i.e. $T$ is order bounded and $|x|\wedge|y|=0$ implies $|x|\wedge|Ty|=0$.
\item $T$ is in the center of $X$, i.e. there is some scalar $\lambda>0$ such that $|Tx|\leq \lambda |x|$ for every $x\in X$.
\end{enumerate}

We say that a linear map $T : X \to X$ is $\vr$-band preserving ($\vr$-BP in short) if,
for any $x \in X$,
$$
\sup \{ \| |Tx| \wedge y \| : y \geq 0, y \perp x \}
 \leq \vr \|x\| .
$$

Our main concern  is to study when an $\vr$-BP operator is a small perturbation of a band-preserving operator. That is, given a Banach lattice $X$ and an $\vr$-$\BP$ operator $T \in B(X)$, when can one find a band-preserving $S \in B(X)$ so that $\|T-S\| \leq \phi(\vr)$ for some function $\phi$ satisfying $\phi(\vr)\rightarrow 0$ as $\vr\rightarrow0$?

We are also interested in quantitative versions of some well known facts concerning band-preserving operators. For instance, on a $\sigma$-Dedekind complete Banach lattice an operator is band-preserving if and only if it commutes with every band projection. A version of this result in terms of the size of the commutators $[T,P]$ where $P$ is a band projection and $T$ is $\vr$-BP will be given in Proposition \ref{p:commutator}. As a consequence, we obtain a quantitative version of another stability property of band-preserving operators, due to C. B. Huijsmans and B. de Pagter, that the inverse of a bijective BP operator is also BP \cite{HdP} (see Corollary \ref{c:inverse}).

The paper is organized as follows: A discussion of several properties of
almost band preserving operators as well as some equivalent characterizations
of this class can be found in Section \ref{s:band}. Recall that a band preserving map
on a Banach lattice is always bounded \cite{AVK}. Motivated by this fact,
in Section \ref{s:auto}, we will study the automatic continuity of $\vr$-BP maps.
Almost central operators and their connection with $\vr$-BP operators
will be studied in Section \ref{s:center}. In Section \ref{s:stability},
we prove that any $\vr$-BP map on a Banach lattice $X$ is a small perturbation
of a BP one, provided $X$ is order continuous (Theorem \ref{t:BP ocont})
or has Fatou norm (Proposition \ref{p:ocont}).

Section \ref{s:C(K)} contains a similar result for $\vr$-BP maps on $C(K)$ spaces
(Theorem \ref{t:C(K)}).
In Section \ref{s:counter}, we present an example of a Banach lattice $E$
with the property that, for every $\vr >0$, there exists an $\vr$-BP contraction
$T \in B(E)$ whose distance from the set of BP maps is larger than
$1/2$ (Proposition \ref{p:better counter}).

Throughout, we use standard Banach lattice terminology and notation.
For more information we refer the reader to the monographs \cite{A-B} or \cite{M-N}.
The closed unit ball of a normed space $Z$ is denoted by $\ball(Z)$.

\section{Basic properties of almost band-preserving operators}\label{s:band}

\begin{definition}\label{d:almostBP}
Given a Banach lattice $X$, we say that a linear mapping $T:X\rightarrow X$ is $\vr$-band preserving ($\vr$-BP) if, for any $x \in X$,
$$
\sup \{ \| |Tx| \wedge y \| : y \geq 0, y \perp x \} \leq \vr \|x\| .
$$
\end{definition}

Observe that every bounded operator $T:X\rightarrow X$ is trivially $\|T\|$-BP. Thus, for a bounded operator $T$, $\vr$-BP is meaningful only for $\vr<\|T\|$. Note that if two operators $T_1,T_2\in B(X)$ are such that $T_i$ is $\vr_i$-BP for $i=1,2$, then $T_1+T_2$ is $(\vr_1+\vr_2)$-BP. Similarly, if $T\in B(X)$ is $\vr$-BP, then, for any scalar $\lambda$, $\lambda T$ is $|\lambda|\vr$-BP.

In order to reformulate Definition \ref{d:almostBP} in the language of bands, we need to recall the notion of band projection. A band $Y$ of a Banach lattice $X$ is called a projection band if $X=Y\oplus Y^\perp$, where
$$
Y^\perp=\{x\in X:|x|\wedge |y|=0 \textrm{ for every } y\in Y\}.
$$
Several facts which arise for spaces with an unconditional basis can be generalized to more general Banach lattices by means of projection bands (see \cite[1.a]{LT2}).

A characterization of projection bands can be found in \cite[Proposition 1.a.10]{LT2}. In particular, if $X$ is a $\sigma$-Dedekind complete Banach lattice (i.e. every bounded sequence has a supremum and an infimum), then for each $x\in X_+$ we can consider the principal band projection $P_x$ given by
$$
P_x(z)=\bigvee_{n=1}^\infty (nx\wedge z)
$$
for $z\in X_+$, and extended linearly as $P_x(z)=P_x(z_+)-P_x(z_-)$ for a general $z\in X$. This defines a projection onto the principal band generated by $x$. Recall that a Banach lattice $X$ is said to have the \emph{Principal Projection Property} (\emph{PPP} for short) if every principal band (a band generated by a single element) is a projection band. By \cite[pp.~17-18]{M-N}, a Banach lattice has the PPP if and only if it is $\sigma$-Dedekind complete.

The study of projection bands was initiated in the classical work of S. Kakutani \cite{Kak} concerning concrete representations of Banach lattices. For properties of band projections, see \cite[Section 1.2]{M-N}. Also recall that, by \cite[Proposition 2.4.4]{M-N}, if $X$ is order continuous, then every closed ideal in $X$ is a projection band. The next proposition gives some equivalent reformulations of the definition of $\vr$-BP operator by means of band projections.

\begin{proposition}\label{p:equiv}
Given a Banach lattice $X$ and an operator $T : X \to X$, consider the following statements:
\begin{enumerate}
\item
$T$ is $\vr$-BP.
\item
For any band projection $P$ and any $x \in X$,
$\|PTx\| \leq \vr\|x\|$ whenever $Px = 0$.
\item
For any principal band projection $P$ and any $x \in X$,
$\|PTx\| \leq \vr\|x\|$ whenever $Px = 0$.
\end{enumerate}
Then $(1)\Rightarrow (2)\Rightarrow (3)$. Moreover, if $X$ is $\sigma$-Dedekind complete, then $(3)\Rightarrow(1)$.
\end{proposition}

We need a lemma, which may be known to experts (although we haven't found it in the literature).

\begin{lemma}\label{l:band proj}
Suppose $X$ is a Banach lattice, $x$ is a non-zero element of $X$,
and $P \in B(X)$ is a band projection. Then the following
are equivalent:
\begin{enumerate}
\item
$P x = 0$.
\item
$P |x| = 0$.
\item
$x \perp P(X)$.
\end{enumerate}
\end{lemma}

\begin{proof}
$(1) \Rightarrow (2)$: write $x = x_+ - x_-$. As $0 \leq Py \leq y$
for any $y \in X_+$, we conclude that $0 \leq P x_+ \leq x_+$, and
$0 \leq P x_- \leq x_-$. Consequently $P x_+$ and $P x_-$ are disjoint,
hence $P |x| = P x_+ + P x_- = |Px|$. If $P x = 0$, then
$P |x| = |P x| = 0$. %The converse implication is proved similarly.

$(2) \Rightarrow (3)$: It suffices to show that
$x \wedge z = 0$ whenever $x, z \in X_+$ satisfy
$Px = 0$ and $Pz = z$. Let $u = x \wedge z$. Then
$Pu \leq Px = 0$. On the other hand, $P$ is a band projection.
As $z$ belongs to the band $P(X)$, the same must be true for $u$,
hence $u=Pu=0$.

$(3) \Rightarrow (1)$: By \cite[Lemma 1.2.8]{M-N},
$I-P$ is the band projection onto $P(X)^\perp$.
Then $(I-P)x = x$, hence $Px = 0$.
\end{proof}

\begin{proof}[Proof of Proposition \ref{p:equiv}]
$(1) \Rightarrow (2)$:
Fix a norm one $x \in X$, and suppose $P$ is a band projection
so that $Px = 0$. We have to show that $\|P T x\| \leq \vr$.
If $PTx = 0$, we are done. Otherwise, let $y = |PTx|$.
By Lemma \ref{l:band proj}, $P |x| = 0$, and $x \perp P(X)$.
Moreover, by the proof of that lemma, $y = P |Tx| \perp x$.
Write $|Tx| = P |Tx| + (I-P)|Tx|$.
The ranges of $P$ and $I-P$ are mutually disjoint bands.
Therefore,
$$
|Tx| \wedge y = (P |Tx| + (I-P)|Tx|) \wedge y = (P|Tx|) \wedge y=P(|Tx|).
$$
Indeed,
$$
(P |Tx| + (I-P)|Tx|) \wedge y \leq
(P |Tx|) \wedge y + ((I-P)|Tx|) \wedge y = (P|Tx|) \wedge y ,
$$
and on the other hand, by the positivity of $I-P$,
$(P |Tx| + (I-P)|Tx|) \wedge y \geq (P|Tx|) \wedge y$.
From (1), it follows that $\|P |Tx| \| \leq \vr \|x\|$.

$(2) \Rightarrow (3)$ is clear.

$(3) \Rightarrow (1)$: Assume that $X$ is $\sigma$-Dedekind compete. For every $x\in X$ we can consider $P_x$ the band projection onto the band generated by $x$. Suppose $y$ is positive, and disjoint from $x$.
Then $Q = I-P_x$ is a band projection, with $Qy = y$.
As shown above,
$$
|Tx| \wedge y = (Q|Tx|) \wedge y,
$$
hence $\||Tx| \wedge y\| \leq \|Q|Tx|\|$.
However, $Q |Tx| = |QTx|$, hence, by (3),
$$
\||Tx| \wedge y\| \leq \|QTx\| \leq \vr \|x\|.
$$
\end{proof}

Recall that if $E$ has the PPP, then $T \in B(E)$ is BP if and only if it commutes with any band projection \cite[Proposition 3.1.3]{M-N}.
Given operators $S,T\in B(E)$, we consider their commutator $[S,T]=ST-TS$.

\begin{proposition}\label{p:commutator}
Let $T$ be an operator on a $\sigma$-Dedekind complete Banach lattice. \begin{enumerate}
\item If for every band projection $P$, $\|[P,T]\|\leq\vr$, then $T$ is $\vr$-BP.
\item If $T$ is $\vr$-BP then for any band projection $P$, $\|[P,T]\|\leq2\vr$.
\end{enumerate}
\end{proposition}

\begin{proof}
We will use the equivalence with $(2)$ in Proposition \ref{p:equiv}. Suppose first that for every band projection $P$, $\|[P,T]\|\leq\vr$. Let $Q$ be a band projection and $x$ be such that $Qx=0$, then we have
$$
\|QTx\|=\|(QT-TQ)x\|\leq\|[Q,T]\|\|x\|\leq\vr\|x\|.
$$

For the second statement, given a band projection $P$, let $P^\perp$ denote its orthogonal band projection. For $x\in X$ we have
\begin{align*}
\|(PT-TP)x\|&=\|(PT-TP)(Px+P^\perp x)\|\\
&=\|PTPx-TPx+PTP^\perp x-TPP^\perp x\|\\
&=\|-P^\perp TPx+PTP^\perp x\|\\
&\leq \|P^\perp TPx\|+\|PTP^\perp x\|
\end{align*}
Now, since $P^\perp Px=PP^\perp x=0$ we get that $\|(PT-TP)x\|\leq2\vr\|x\|$, i.e. $\|[P,T]\|\leq2\vr$.
\end{proof}

The following is a version of the result in \cite{HdP} that the inverse of a bijective BP operator is also BP.

\begin{corollary}\label{c:inverse}
Let $X$ be a $\sigma$-Dedekind complete Banach lattice.
If $T \in B(X)$ is invertible and $\vr$-BP,
then $T^{-1}$ is $(2\|T^{-1}\|^2\vr)$-BP.
\end{corollary}

\begin{proof}
Let $P$ be any band projection. By Proposition \ref{p:commutator}, we have $\|PT-TP\|\leq2\vr$. Therefore,
$$
\|T^{-1}P-PT^{-1}\|=\|T^{-1}(PT-TP)T^{-1}\|\leq2\|T^{-1}\|^2\vr,
$$
and the result follows by Proposition \ref{p:commutator} again.
\end{proof}

\begin{remark}\label{r:inverse}
In general, the $\|T^{-1}\|^2$ factor cannot be avoided in Corollary \ref{c:inverse},
even when $T$ is positive. Indeed, consider
$$
T =  \begin{pmatrix} 0 & \vr \\ \vr & 0 \end{pmatrix},
$$
acting on $\ell^p_2$ (with $1 \leq p \leq \infty$). Clearly $\|T\| = \vr$, hence $T$ is $\vr$-BP. However, $T^{-1}$ cannot be $c$-BP for $c < 1/\vr$. Indeed, suppose $T^{-1}$ is $c$-BP, then we have
$$
c\vr\geq\left\|T^{-1}\begin{pmatrix} \vr \\ 0 \end{pmatrix}\wedge\begin{pmatrix} 0 \\ 1 \end{pmatrix}\right\|=1.
$$
Thus, Corollary \ref{c:inverse} is sharp (up to a constant independent of $\|T^{-1}\|$).
\end{remark}

Below we show that any band-preserving operator on a K\"othe
function space is a multiplication operator.
Recall that a K\"othe function space on a $\sigma$-finite measure space
$(\Omega,\Sigma,\mu)$ is a Banach space $X$ consisting of equivalence classes,
modulo equality almost everywhere, of locally integrable functions on $\Omega$ such that:
\begin{enumerate}
\item
If $|f(\omega)|\leq|g(\omega)|$ holds a.e. on $\Omega$ with $f$ $\Sigma$-measurable and
$g\in X$, then $f\in X$ and $\|f\|\leq\|g\|$.
\item
If $S \in \Sigma$, and $\mu(S) \in (0,\infty)$, then $\chi_S \in X$.
\end{enumerate}

\begin{proposition}\label{p:ABP Kothe}
Suppose $E$ is a K\"othe function space on a Borel measure space $(\Omega,\mu)$,
and $T \in B(E)$ is band preserving. Then there exists $\phi \in L_\infty(\mu)$
so that $Tf = \phi f$ for any $f \in E$.
\end{proposition}

A similar result was established in \cite{Zaa}, and our proof
is similar to the one given there.

\begin{proof}
Suppose $T \in B(E(\Omega,\mu))$ is band-preserving. By \cite[Theorem 3.1.5]{M-N},
$T$ is regular, hence (in the terminology of \cite[Section 3.1]{M-N})
an orthomorphism.

First suppose the measure $\mu$ is finite. Let $\one = \one_\Omega$, and
set $\phi = T \one$. Note that $\phi$ is essentially bounded,
with $\|\phi\|_\infty \leq \|T\|$. Indeed, suppose otherwise, and find a set $S$
of positive measure so that $|\phi| > \|T\|\|\one\|$ $\mu$-a.e.~on $S$.
However,
$$
\chi_S \phi = \chi_S \big[ T (\chi_S + \chi_{\Omega \backslash S} ) \big] =
\chi_S \big[ T \chi_S \big] + \chi_S \big[ T \chi_{\Omega \backslash S} \big] .
$$
As $T$ is band-preserving, the last term vanishes. Thus, 
$$
\|T\chi_S\| =\|\chi_S(T\chi_S)\|=\|\chi_S\phi\|> \|T\| \|\chi_S\|,
$$
which is a contradiction.

Define the operator $S \in B(E)$ via $S f = \phi f$.
Then $S \one = T \one$. As $\one$ generates $E$ as a band,
\cite[Proposition 3.1.6]{M-N} implies $T = S$.

Now suppose $\mu$ is $\sigma$-finite. Represent $\Omega$ as an increasing
union of the sets $\Omega_i$, such that $\mu(\Omega_i)$ is finite.
Let $\mu_i = \mu|_{\Omega_i}$, and define the operator $T_i \in B(E(\Omega_i,\mu_i))$
via $T_i f = (T f) \one_{\Omega_i}$. Clearly $T_i$ is band-preserving, hence
by the above, there exists $\phi_i \in L_\infty(\mu)$, supported on $\Omega_i$,
so that $T_i f = \phi_i f$ for any $f$. It is clear that
$\phi_{i+1} \one_{\Omega_i} = \phi_i$. Due to the boundedness of $T$,
$\phi \in L_\infty(\mu)$.
\end{proof}

\begin{remark}
By \cite[Theorem 3.1.12]{M-N}, any band-preserving operator $T$ on a Banach lattice satisfies $-\|T\| I \leq T \leq \|T\| I$.
% This fact may be used to streamline the proof of Proposition \ref{p:ABP Kothe},
% although I do not see how.
\end{remark}

Recall that an ideal $U$ in a Banach lattice $X$ is a subspace
with the property that $y\in U$ whenever $|y|\leq |x|$ and $x\in U$.
One might consider ideal preserving operators $T:X\rightarrow X$, i.e.
those satisfying that for every (closed) ideal $U\subset X$, $T(U)\subset U$.
However, this notion is actually equivalent to that of band preserving operator:
since a band is also an ideal, every ideal preserving operator is in particular
band preserving; on the other hand, if $T$ is a band preserving operator, then $|Tx|\leq \|T\| |x|$, and this in turn implies
that $T$ is ideal preserving.

We show that the same holds for ``almost'' band preserving and ideal preserving maps.

\begin{definition}
Given a Banach lattice $X$, a linear map $T:X\rightarrow X$ is $\vr$-ideal preserving
($\vr$-IP, for short) if, for every ideal $U\subset X$ and $x\in\ball(X) \cap U$,
there exist $y\in U$ and $z\in X$ with $\|z\|\leq\vr$, such that $Tx=y+z$.
\end{definition}

\begin{theorem}\label{t:BP=>IP}
Suppose $X$ is a Banach lattice, $T : X \to X$ is a linear map, and $\vr > 0$.
Consider the following statements:
\begin{enumerate}
\item
$T$ is $\vr$-BP.
\item
If $x \in \ball(X)$ and $x^* \in \ball(X^*)$ satisfy
$\langle |x^*|, |x| \rangle = 0$, then $|\langle x^*, Tx \rangle| \leq \vr$.
\item
For every $\vr^\prime>\vr$, $T$ is $\vr^\prime$-IP.
\item
For every $x\in \ball(X)$ and $\vr'>\vr$, there is $\lambda>0$
such that $|Tx|\leq\lambda|x|+z$ for some $\|z\|\leq \vr'$.
\end{enumerate}
Then $(2)\Leftrightarrow (3)\Leftrightarrow (4)\Rightarrow (1)$.
If, in addition, $T$ is bounded, then $(1)\Rightarrow (2)$
(that is, all the four statements are equivalent).
\end{theorem}

\begin{proof}
$(1)\Rightarrow (2)$, for $T$ bounded:
Suppose, for the sake of contradiction, that $(1)$ holds, but $(2)$ doesn't.
Then there exist $x \in \ball(X)$ and $x^* \in \ball(X^*)$
so that $\langle|x^*|,|x|\rangle=0$ and
$$
\langle |x^*| , |Tx| \rangle \geq \big| \langle x^* , Tx \rangle \big| = c > \vr .
$$
Pick $\delta > 0$ so that $\vr + \|T\| (\|T\|+1) \delta < c$.

For brevity of notation, let $x^\prime = |Tx|$.
We find $y, y^\prime \in X$ so that $|y| \leq |x|$,
$\|x - y\| \leq (\|T\|+1) \delta$,
$y^\prime \geq 0$, $y^\prime \perp y$, and
$\|x^\prime \wedge y^\prime \| \geq c$. Once this is achieved,
the inequality
$$
\big\| |Ty| \wedge y^\prime \big\| \geq
\big\| |Tx| \wedge y^\prime \big\| - \|T\| \|x-y\| > \vr
$$
will give the desired contradiction.

Consider the (not necessarily closed) ideal $I \subset X$, generated by
$x_0 = x^\prime \vee |x|$. In a canonical fashion, we find a bijective
lattice homomorphism $j : C(\Omega)\to I$, where $\Omega$ is a Hausdorff
compact (so $j\one=x_0$). We have
$$
K := \|j\|_{B(C(\Omega),X)} = \|j \one\|_X
 \leq \|x\| + \|x^\prime\| \leq \|T\| + 1 .
$$
Let $\phi$ and $\phi^\prime$ in $C(\Omega)$ such that $x=j(\phi)$ and $x^\prime=j(\phi^\prime)$. % Then
Set $\psi = (\phi_+ - \delta \one)_+ - (\phi_- - \delta \one)_+$,
and $y = j (\psi)$. Then $|\psi| \leq |\phi|$, hence $|y| \leq |x|$.
Furthermore, $\|\psi - \phi\|_\infty = \delta$, hence $\|x - y\| \leq K \delta$.

Now consider the closed sets
$\Omega_1 = \{ \omega \in \Omega : |\phi(\omega)| \leq \delta/2\}$ and
$\Omega_0 = \{ \omega \in \Omega : |\phi(\omega)| \geq \delta\}$.
By Urysohn's Lemma, there exists $h \in C(\Omega)$ so that $0 \leq h \leq 1$,
$h|_{\Omega_1} = 1$, and $h|_{\Omega_0} = 0$.

Consider $\mu  \in C(\Omega)^*$ given by
$\langle \mu , f \rangle = \langle |x^*| , j (f) \rangle$
for $f \in C(\Omega)$. Clearly $\mu$ is a positive measure, and
$\langle \mu, |\phi| \rangle = \langle |x^*|, |x| \rangle = 0$.

Now set $\psi^\prime = \phi^\prime h$. Note that $1-h \leq 2 \delta^{-1} |\phi|$,
hence $\langle \mu, (1-h)\eta \rangle = 0$ for any $\eta \in C(\Omega)$.
Consequently,
% By the above,
$\langle \mu, \psi^\prime \rangle = \langle \mu, \phi^\prime \rangle \geq c$.
Set $y^\prime = j (\psi^\prime)$. As $\psi \perp \psi^\prime$,
we also have $y \perp y^\prime$. Further,
$0 \leq \psi^\prime \leq \phi^\prime$, hence $0 \leq y^\prime \leq x^\prime$.
Consequently,
$$
\| x^\prime \wedge y^\prime \| = \|y^\prime\| \geq \langle |x^*|, y^\prime \rangle =
\langle \mu , \psi^\prime \rangle = \langle \mu , \phi^\prime \rangle \geq c ,
$$
which is the desired result.
\smallskip

$(2)\Rightarrow(3)$: Suppose that $(3)$ does not hold.
Then there exist $\vr'>\vr$, an ideal $U\subset X$ and $x\in \ball(U)$
such that for every $y\in U$, $\|Tx-y\|>\vr'$.
We can and do assume $U$ is closed. By Hahn-Banach Theorem,
there is $x^*\in \ball(X^*)$ such that $\langle x^*, y\rangle=0$
for every $y\in U$ and
$$
\langle x^*,Tx\rangle={\mathrm{dist}}(Tx, U)\geq\vr'>\vr.
$$
As $U$ is an ideal, we have
$$
\langle |x^*|,|x|\rangle=\sup\{|\langle x^*,y\rangle|:|y|\leq |x|\}=0.
$$
This is impossible if $(2)$ holds. % when $T$ is $\vr$-IP.
\smallskip

$(3)\Rightarrow (4)$ is immediate by considering the principal ideal generated by $x$.
\smallskip

$(4)\Rightarrow (2)$: Suppose $x^*$ and $x$ are as in (2). Fix $\vr^\prime > \vr$
and find $\lambda$ s.t. $|Tx| \leq \lambda |x| + z$, with $\|z\| \leq \vr^\prime$.
Then
$$
|\langle x^*, Tx \rangle| \leq \langle x^*, |Tx| \rangle \leq
\langle x^*, z \rangle \leq \|z\| \leq \vr^\prime .
$$
As $\vr^\prime$ can be arbitrarily close to $\vr$, we obtain
$|\langle x^*, Tx \rangle| \leq \vr$.
\smallskip

$(4)\Rightarrow (1)$:
Pick disjoint $x \in \ball(X)$ and $y \in X_+$. For every $\vr'>\vr>0$, by $(4)$, there exist $\lambda>0$ and $z\in X$ with $\|z\|\leq \vr'$ and $|Tx|\leq\lambda|x|+z$. Without loss of generality we can take $z\geq0$. It follows that
$$
\| |Tx| \wedge y \| \leq \| (\lambda |x|+z) \wedge y \| \leq \| z \wedge y \| \leq \vr'.
$$
Since this holds for arbitrary $\vr'>\vr$ we get that $T$ is $\vr$-BP.
\end{proof}

Notice the following fact concerning the duality of band projections and almost band preserving operators. This will be useful in the stability results of Section \ref{s:stability}.

\begin{proposition}\label{p:duality} Let $X$ be a Banach lattice:\begin{enumerate}
\item If $P$ is a band projection on $X$, then $P^*$ is a band projection on $X^*$.
\item If $X$ is order continuous and $P$ is a band projection on $X^*$, then $P^*|_X$ is a band projection on $X$.
\item If $T\in B(X)$ is such that $T^*$ is $\varepsilon$-BP, then $T$ is $\varepsilon$-BP.
\item If $X$ is order continuous and $T\in B(X)$ is $\varepsilon$-BP, then $T^*$ is $\varepsilon$-BP.
\end{enumerate}
\end{proposition}

\begin{proof}
(1): This is a direct consequence of the fact that $P$ is a band projection if and only if $P^2=P$ and $0\leq P\leq I$.

(2): By part (1), we know that $P^*$ is a band projection on $X^{**}$. For $x\in X$,
we have that $|P^*x|\leq |x|$, and since $X$ is an ideal in $X^{**}$ \cite[Theorem 2.4.2]{M-N},
it follows that $P^*x\in X$. Thus, $P^*|_X$ is a band projection on $X$.

(3): Let $x\perp y$ in $X$ with $\|x\|\leq1$ and $y\geq0$, we have to show that $\||Tx|\wedge y\|\leq\vr$. First, since $x\perp(|Tx|\wedge y)$, by \cite[Lemma 1.4.3]{M-N}, we can find $y^* \in \ball(X^*)$ so that $\langle |y^*|, |x| \rangle = 0$,
and $\langle |y^*|, |Tx| \wedge y \rangle = \| |Tx| \wedge y \|$.
Passing to $|y^*|$, we can and do assume $y^* \geq 0$.

Thus, it suffices to prove that $\langle y^*, |Tx| \rangle \leq \vr$.
By \cite[Lemma 1.4.4]{M-N},
$$
\langle y^*, |Tx| \rangle = \max \big\{ |\langle z^*, Tx \rangle| : |z^*| \leq y^* \big\} .
$$
For any such $z^*$, we have $\langle |z^*|, |x| \rangle = 0$. Therefore, $x$ annihilates on the principal ideal generated by $z^*$. Moreover, since the elements of $X$ acting on $X^*$ are order continuous functionals (cf. \cite[p. 61]{A-B}), it follows that $x$ also annihilates on the band generated by $z^*$.

Let $P$ denote the band projection onto the band generated by $z^*$. By the above, we have $P^*x=0$. Therefore, by Proposition \ref{p:equiv} we get
$$
|\langle T^*z^*,x\rangle|=|\langle P^\perp T^*z^*,x\rangle|\leq \|P^\perp T^*z^*\|\leq \vr.
$$

(4): Suppose $T$ is $\varepsilon$-BP, and let $P$ be a band projection on $X^*$.
Since $X^*$ is $\sigma$-Dedekind complete, by Proposition \ref{p:equiv},
it is enough to show that for every $x^*\in X^*$ such that $Px^*=0$
we have $\|PT^*x\|\leq\varepsilon\|x^*\|.$  According to (2), there is
a band projection on $X$, given by $Q=P^*|_X$ such that $Q^*=P$.
Let $Q^\perp=I-Q$ be the band projection onto the complementary band.
Since $T$ is $\varepsilon$-BP and $Q^\perp Qx=0$, by Proposition \ref{p:equiv},
we have that $\|Q^\perp TQx\|\leq\varepsilon\|Qx\|\leq\varepsilon\|x\|.$
Now, using the fact that $Px^*=0$, we have
\begin{eqnarray*}
\|PT^* x^*\|&=&\sup\{\langle PT^*x^*,x\rangle:x\in X,\,\|x\|\leq1\}\\
&=&\sup\{\langle x^*,TQx\rangle:x\in X,\,\|x\|\leq1\}\\
&=&\sup\{\langle x^*,(I-Q)TQx\rangle:x\in X,\,\|x\|\leq1\}\\
&\leq&\sup\{\|Q^\perp TQx\|\|x^*\|:x\in X,\,\|x\|\leq1\}\leq\varepsilon\|x^*\|,
\end{eqnarray*}
as desired.
\end{proof}

It is well known that any band-preserving operator is also disjointness preserving.
For $\vr$-BP maps, a similar result holds. Recall that an operator between Banach lattices
$T:X\rightarrow Y$ is $\vr$-disjointness preserving if $\||Tx|\wedge|Ty|\|\leq\vr$ whenever
$x,y\in \ball(X)$ satisfy $x\perp y$. This class of operators has been the object of research in \cite{OikTra}, where it was studied whether an $\vr$-disjointness preserving operator can always be approximated by a disjointness preserving one.

\begin{proposition}\label{p:BP=>DP}
If $X$ is a Banach lattice, and $T \in B(X)$ is $\vr$-BP,
then $T$ is $2\vr$-disjointness preserving.
\end{proposition}

\begin{proof}
We have to show that, if $x$ and $y$ are disjoint elements of $\ball(X)$, then
$\| |Tx| \wedge |Ty| \| \leq 2 \vr$. We can and do assume that the Banach
lattice $X$ is separable. Indeed, it is easy to see that any separable subset of $X$
is contained in a separable sublattice invariant under $T$.

As $X$ is separable, it contains a quasi-interior point, which we call $e$:
we have $e \geq 0$, and $z = \lim_m z \wedge me$ for any $z \in X_+$.
Note that, for any $z \in X$, and $\delta > 0$, we can find $\tilde{z} \in X$
and $m \in \NN$ so that $|\tilde{z}| \leq me$, and $\|z - \tilde{z}\| \leq \delta$.
Indeed, write $z = z_+ - z_-$, and find $m$ so that, for $\sigma = \pm$,
$\|z_\sigma - z_\sigma \wedge me\| < \delta/2$. Then
$\tilde{z} = z_+ \wedge me - z_- \wedge me$ has the required properties.

Due to the continuity of $T$, we shall henceforth assume that $|x| \vee |y| \leq me$
for some $m \in \NN$. Further, by changing the quasi-interior point $e$,
we can assume $|x| \vee |y| \leq e$.

Now fix $c > 0$, and let $x^\prime = (x_+ - ce)_+ - (x_- - ce)_+$. Note
that $x_+ - ce \leq (x_+ - ce)_+ \leq x_+$, and similar inequalities holds
for $(x_- - ce)_+$. Therefore,
$$
\| x - x^\prime \| \leq \|x_+ - (x_+ - ce)_+\| + \|x_- - (x_- - ce)_+\| \leq 2c\|e\| .
$$
Analogously, we define $y^\prime= (y_+ - ce)_+ - (y_- - ce)_+$, which satisfies
$\|y - y^\prime\| < 2c\|e\|$. As $c$ can be arbitrarily small, it suffices
to show that, for any $n \in \NN$,
\begin{equation}
\label{eq:norm est}
\| |Tx^\prime| \wedge |Ty^\prime| \wedge ne \| \leq 2\vr .
\end{equation}

Let $a = n(e - |x|/c)_+$ and $b = n(e - |y|/c)_+$.
Viewing the elements we are working with as elements of the ideal
generated by $e$ (which can, in turn, be identified with $C(K)$),
we see that $a \perp x^\prime$ and $b \perp y^\prime$.
As $T$ is $\vr$-BP, we have $\| |Tx^\prime| \wedge a \| \leq \vr$
and $\| |Ty^\prime| \wedge b \| \leq \vr$. We have
$$
|Tx^\prime| \wedge |Ty^\prime| \wedge (a+b) \leq |Tx^\prime| \wedge a + |Ty^\prime| \wedge b .
$$
The inequality \eqref{eq:norm est} now follows from $a + b \geq ne$.
\end{proof}

We do not know whether the continuity of $T$ is actually necessary in Proposition \ref{p:BP=>DP}. However, for $\sigma$-Dedekind complete spaces we have:

\begin{proposition}\label{p:BP=>DP_nocont}
If $X$ is a $\sigma$-Dedekind complete Banach lattice, and $T$ is an $\vr$-BP linear map,
then $T$ is also $3\vr$-disjointness preserving.
\end{proposition}

\begin{proof}
Suppose $x$ and $y$ are disjoint elements in the unit ball of $X$. Then
\begin{align*}
|Tx| \wedge |Ty|
& =
(P_x |Tx| + P_x^\perp |Tx|) \wedge (P_y |Ty| + P_y^\perp |Ty|)
\\ & \leq
(P_x |Tx|) \wedge (P_y |Ty|) + (P_x^\perp |Tx|) \wedge (P_y |Ty|)
\\ & +
(P_x |Tx|) \wedge (P_y^\perp |Ty|) +
(P_x^\perp |Tx|) \wedge (P_y^\perp |Ty|) .
\end{align*}
By the triangle inequality,
$\| |Tx| \wedge |Ty| \| \leq 3 \vr$.
\end{proof}

Any disjointness preserving operator (hence also any band-preserving operator) is regular \cite[Theorem 3.1.5]{M-N}. Moreover,
if $T \in B(X)$ is band-preserving, then so is $|T|$. One might wonder whether the modulus of a regular $\vr$-BP operator is also $\vr$-BP. This is the case for AM-spaces and AL-spaces. Recall that a Banach lattice is an AL-space if $\|x+y\|=\|x\|+\|y\|$ whenever $x\wedge y=0$; an AM-space if $\|x+y\|=\max\{\|x\|,\|y\|\}$ whenever $x\wedge y=0$. AL-spaces are order isometric to spaces $L_1(\mu)$, while AM-spaces are order isometric to sublattices of spaces $C(K)$ \cite[1.b]{LT2}.

\begin{proposition}\label{p:modulus}
Suppose $T \in B(X)$ is a $\vr$-BP operator.
\begin{enumerate}
\item
If $X$ is an AM-space, and $T$ is regular, then $|T|$ is $\vr$-BP.
\item
If $X$ is an AL-space, then $|T|$ is $\vr$-BP.
\end{enumerate}
% If $T$ is $\vr$-BP, then $|T|$ is $\vr$-BP as well.
\end{proposition}

\begin{proof}
(1) $X$ is an AM-space. Given $x,y\in X$ with $x\perp y$ we have
\begin{eqnarray*}
\big\| \big||T|x\big|\wedge |y|\big\|&\leq& \big\| |T||x|\wedge |y|\big\|=\Big\|\Big(\bigvee_{|z|\leq |x|}|Tz|\Big)\wedge|y|\Big\|\\
&=&\Big\|\bigvee_{|z|\leq |x|}(|Tz|\wedge|y|)\Big\|=\bigvee_{|z|\leq |x|}\Big\||Tz|\wedge|y|\Big\|\\
&\leq&\bigvee_{|z|\leq |x|}\vr\|z\|=\vr\|x\|,
\end{eqnarray*}
where the last inequality follows from the fact that $z\perp y$ for every $|z|\leq|x|$ and $T$ is $\vr$-BP.

(2) $X$ is an AL-space, so in particular it is order continuous. By Proposition \ref{p:duality}, we have that $T^*$ is $\vr$-BP. Also, note that every operator on an AL-space is regular \cite[Theorem 4.75]{A-B}. By \cite[Proposition 1.4.17]{M-N}, we have that $|T^*|=|T|^*$. Since $X^*$ is an AM-space, by part (1) we get that $|T|^*$ is $\vr$-BP. Again, Proposition \ref{p:duality} yields that $|T|$ is $\vr$-BP, as claimed.
\end{proof}

\begin{remark}\label{r:BP modulus}
Proposition \ref{p:modulus} fails for general Banach lattices.
For every $\vr > 0$ there exists a regular $\vr$-BP contraction $T \in B(\ell_2)$
so that $|T|$ is not $c$-BP whenever $c < 1/2$. An example can be found in
\cite[Proposition 9.4]{OikTra}. We briefly outline the construction.

For $i \in \NN$ let $S_i$ be the $2^i \times 2^i$ Walsh unitary, and set
$T_i = I_{\ell_2^{2^i}} + 2^{-i/2} S_i \in B(\ell_2^{2^i})$.
For $\vr > 0$, find $n \in \NN$ so that $2^{-n/2} < \vr$.
Let $T = \oplus_{i \geq n} T_i$ be an operator on
$E = (\sum_{i=n}^\infty \ell_2^{2^i})_2$ (this space can be identified with $\ell_2$).
Clearly $\|T - I_E\| < \vr$, hence $T$ is $\vr$-BP.
However, as in \cite[Proposition 9.4]{OikTra}, $|T| = \oplus_i |T_i|$, where
$|T_i| = I_{\ell_2^{2^i}} + \xi_i \otimes \xi_i$, with
$\xi_i = 2^{-i/2} (1, \ldots, 1)$ is a unit vector in $\ell_2^{2^i}$.
Taking $x = 2^{(1-i)/2} (1, \ldots, 1, 0, \ldots, 0)$ and
$y = (0, \ldots, 0, 1, \ldots, 1)$ (both strings contain an equal number of
$0$'s and $1$'s), we see that $\| |T_i| x \wedge y \| = 1/2$.

One can use the same reasoning to construct, for $1 < p < \infty$ and $\vr > 0$,
a regular $\vr$-BP operator $T \in B(\ell_p)$ so that $|T|$ is $c$-BP only
when $c \geq c_p$, where $c_p > 0$ depends on $p$ only.
\end{remark}

\section{Automatic continuity}\label{s:auto}

In certain situations, $\vr$-BP linear maps are automatically continuous.

\subsection{K\"othe spaces}\label{ss:auto kothe}
Recall that a Banach lattice $X$ has \emph{Fatou norm with constant $\fatou$} if,
for any non-negative increasing net $(x_i) \subset X$, with $\sup_i \|x_i\| < \infty$,
and $\vee_i x_i \in X$, we have $\| \vee_i x_i \| \leq {\fatou} \sup_i \|x_i\|$.
For K\"othe function spaces this is equivalent to the following: if
$f, f_1, f_2, \ldots \in X$ satisfy $f_n(\omega)\uparrow f(\omega)$ a.e.,
with $f_n(\omega)\geq0$ a.e., then $\|f\|=\lim_n\|f_n\|$.
Note that a Banach lattice which has a Fatou norm with constant $\fatou$
admits an equivalent lattice norm which is Fatou with constant $1$.
Indeed, we can set
$$
\triple{x}=\inf\{\sup_i\|x_i\|:|x|=\vee_i x_i,\, x_i \textrm{ increasing, } \sup_i\|x_i\|<\infty\}.
$$
If $(X, \| \cdot \|)$ is a K\"othe function space, then the same is true for
$(X, \triple{ \cdot })$.

\begin{proposition}\label{p:auto}
Suppose $X$ is a K\"othe function space on a $\sigma$-finite measure space $(\Omega,\mu)$,
with Fatou norm.
If $T : X \to X$ is a $\vr$-BP linear map, then $T$ is continuous.
\end{proposition}

Let us first fix some notation. For a measurable $A \subset \Omega$,
denote by $P_A$ the band projection onto the band generated by $A$ (i.e. $P_A x = \chi_A x$),
and set
$$
X_A = P_A(X) = \{x \in X : x = \chi_A x\}.
$$
For any $x \in P_A(X)$, $\|P_{A^c} T x\| \leq \vr \|x\|$.
Indeed, it suffices to apply the definition of $\vr$-BP to
$y = |P_{A^c} x|/\|P_{A^c} x\|$.

For notational convenience, we assign infinite norm to any
unbounded operator. By renorming if necessary, we can assume that the Fatou
constant of $X$ equals $1$.

\begin{lemma}\label{l:seq}
Suppose $X$ and $T$ are as in Proposition \ref{p:auto}, and
$(A_i)_{i \in I}$ is a family of disjoint subsets of $\Omega$,
each having positive measure. Then there exists $C > 0$ so that
$\|T P_{A_i}\| \leq C$ for all but finitely many indices $i\in I$.
\end{lemma}

\begin{proof}
Suppose otherwise. Then we can find a mutually disjoint sequence
$(x_k)$ with $\supp x_k\subset A_k$, so that, for each $k$, $\|x_k\| < 2^{-k}$, and $\|T x_k\| > 2^k$.
Let $x = \sum_{k=1}^\infty x_k$, and $\tilde{x}_k = x - x_k$.
Then
$$
\|Tx\| \geq \|P_{A_k} T (x_k + \tilde{x}_k)\| \geq
 \|P_{A_k} T x_k\| - \|P_{A_k} T \tilde{x}_k\| .
$$
But $\|P_{A_k} T x_k\| \geq \|T x_k\| - \vr \|x_k\| > 2^k - \vr \|x_k\|$,
while $\|P_{A_k} T \tilde{x}_k\| \leq \vr \|\tilde{x_k}\|$. Thus,
$$
\|Tx\| \geq 2^k - \vr(\|x_k\| +\| \tilde{x}_k\|) \geq 2^k - \vr.
$$
This inequality should hold for any $k$, which is impossible.
\end{proof}

\begin{lemma}\label{l:increase}
Suppose $X$ and $T$ are as in Proposition \ref{p:auto}.
If $(A_i)_{i \in \NN}$ is an increasing sequence of measurable subsets
of $\Omega$, so that for each $i\in\mathbb N$, $T P_{A_i}$ is bounded,
then $\sup_n \|T P_{A_n}\| < \infty$.
\end{lemma}

\begin{proof}
Suppose $\sup_n \|T P_{A_n}\| = \infty$.
Then there exist $1 \leq n_0 < n_1 < n_2 < \ldots$ so that
$\|T P_{A_{n_0}}\| > 1$, and $\|T P_{A_{n_{k+1}}}\| > 3 \|T P_{A_{n_k}}\|$
for every $k$. Consequently, $\|T P_{B_k}\| > 2^k$ for every $k$,
where $B_0 = A_{n_0}$, and $B_k = A_{n_k} \backslash A_{n_{k-1}}$
for $k > 0$. This contradicts Lemma \ref{l:seq}.
\end{proof}

\begin{lemma}\label{l:increase2}
In the notation of Lemma \ref{l:increase}, $T P_{\cup_i A_i}$ is bounded.
\end{lemma}

\begin{proof}
Let $A = \cup_i A_i$. Let $C = \sup_i \|T P_{A_i}\|$.
Suppose, for the sake of contradiction, that there exists a norm one
$x \in P_A(X)$, so that $\|Tx\| > C + 2\vr$. Then $\|P_A Tx\| > C+\vr$.
Since $X$ has the Fatou property we have that $\|P_{A_n}Tx\| \to \|P_ATx\|$, thus $\|P_{A_n} Tx\| > C+\vr$
for $n$ large enough. Write $x = y + z$, where $y = P_{A_n} x$ and
$z = P_{A \backslash A_n} x$. We have $\|P_{A_n} T y\| \leq C$
and $\|P_{A_n} T z\| \leq \vr$, hence, by the triangle inequality,
$\|P_{A_n} T x\| \leq C + \vr$, yielding a contradiction.
\end{proof}

\begin{proof}[Proof of Proposition \ref{p:auto}]
Denote by $\Sigma$ the set of all equivalence classes of measurable
subsets of $\Omega$, of positive measure (two sets are equivalent
if the measure of their symmetric difference is $0$). Abusing the
notation slightly, we identify classes with their representatives.
Denote by $\Sigma_b$ the set of all classes ${\mathbf{S}} \in \Sigma$ so
that, for any (equivalently, all) $S \in {\mathbf{S}}$, $T P_S$
is bounded.

Note that $\Sigma_b$ is closed under finite or countably infinite unions.
The finite case is clear. To handle the infinite case, consider
$A_1, A_2, \ldots \in \Sigma_b$, and show that $A = \cup_k A_k \in \Sigma_b$
as well. Without loss of generality, we can assume $A_1, A_2, \ldots$ are
disjoint. By Lemma \ref{l:seq},
there exists $C > 0$ so that $\|T P_{A_i}\| < C$ for any $i$. Replacing now $A_k$ with
$\cup_{i\leq k} A_i$, by Lemmas \ref{l:increase} and
 \ref{l:increase2}, it follows that $\|T P_A\| \leq C + 2 \vr$.

By Zorn's Lemma (and taking the $\sigma$-finiteness of $\mu$ into account),
we see that $\Sigma_b$ contains a maximal element $[A]$. We claim that
$[A] = \Omega$. Indeed, otherwise $T P_B$ is unbounded for any $B \subset A^c$.
If $A^c$ is a union of finitely many atoms, this is clearly impossible.
Otherwise, write $A^c$ as a disjoint union of infinitely many sets $B_i$ of
positive measure. By Lemma \ref{l:seq}, $T P_{B_i}$ is bounded for some $i$
(in fact, for infinitely many $i$'s), hence $A \cup B_i \in \Sigma_b$,
contradicting the maximality of $A$.
\end{proof}

\begin{remark}\label{r:different}
In a similar fashion, one can prove the following: suppose $\vr > 0$, and
$X$ and $Y$ are K\"othe function spaces on $(\Omega,\Sigma,\mu)$.
Suppose a linear map $T : X \to Y$ has the property that, for any $S \in \Sigma$,
and any $x \in X$ satisfying $x = \chi_S x$, we have $\|\chi_{S^c} [Tx]\| \leq \vr \|x\|$.
Then $T$ is continuous.
\end{remark}

\begin{corollary}\label{c:auto OC}
For any $\vr > 0$, any $\vr$-BP linear map on an order continuous
Banach lattice is continuous.
\end{corollary}

\begin{proof}
Suppose a Banach lattice $X$ is order continuous, and $T : X \to X$ is $\vr$-BP.
By the proof of \cite[Proposition 1.a.9]{LT2} (combined with \cite[Proposition 2.4.4]{M-N}),
$X$ can be represented as an unconditional sum of mutually orthogonal projection bands
$(X_\alpha)_{\alpha \in {\mathcal{A}}}$, having a weak order unit.
Denote the corresponding band projections by $P_\alpha$.
For any $x \in X$, $\sum_\alpha P_\alpha x$ has at most countably
many non-zero terms, and converges unconditionally. For $A \subset {\mathcal{A}}$,
 $X_A = \oplus_{\alpha \in A} X_\alpha \subset X$ is the range of the band
projection $P_A = \sum_{\alpha \in A} P_\alpha$ (indeed, $0 \leq P_A \leq I$).
For each $\alpha$, $X_\alpha$ is order isometric to a K\"othe function space
\cite[pp. 25-29]{LT2}.

Suppose, for the sake of contradiction, that $T : X \to X$ is
an unbounded $\vr$-BP map. As the unconditional decomposition of every
$x \in X$ is at most countable, there exists a countable set $B$
so that $P_B T P_B$ is unbounded. Write $B = \{\beta_1, \beta_2, \ldots\}$.

By Proposition \ref{p:auto},
$P_\alpha T P_\alpha$ is bounded for any $\alpha$, hence the same is true for
$T P_\alpha$. Note first that $\sup_\alpha \|T P_\alpha\| < \infty$.
Indeed, otherwise we can find distinct $\alpha_i$ ($i \in \NN$) and
$x_i \in X_{\alpha_i}$ so that $\|x_i\| < 2^{-i}$, but $\|P_{\alpha_i}T x_i\| > 2^i + \vr$.
Let $x = \sum_i x_i$, and $\tilde{x}_i = x - x_i$. Then for each $i$,
$$
\|Tx\|\geq \|P_{\alpha_i} T x\| \geq
\|P_{\alpha_i} T x_i\| - \|P_{\alpha_i} T \tilde{x}_i\| >
2^i + \vr - \vr = 2^i ,
$$
which is impossible.

Furthermore, let $B_n = \{\beta_1, \ldots, \beta_n\}$.
Then $\sup_n \|T P_{B_n}\| < \infty$.
Indeed, otherwise we can find $n_1 < n_2 < \ldots$ so that
there exists $x_k \in \oplus_{i \in B_{n_k} \backslash B_{n_{k-1}}} X_{\beta_i}$
with $\|x_k\| < 2^{-k}$ and $\|T x_k\| > 2^k + \vr$.
Obtain a contradiction by considering $x = \sum_k x_k$
(as in Lemma \ref{l:increase}).

Finally set $C = \sup_n \|T B_n\|$. Pick a norm one $x \in X_B$.
By the order continuity of $X$, $P_{B_n} \to P_B$ point-norm,
hence for every $\delta > 0$ there exists $n$ so that
$\|P_{B_n} T x\| > \|P_B T x\| - \delta$. But
(reasoning as in the proof of Lemma \ref{l:increase2})
$$
\|P_{B_n} T x\| \leq \|P_{B_n} T P_{B_n} x\| + \|P_{B_n} T P_{B \backslash B_n} x\|
\leq C + \vr ,
$$
hence $\|T x\| \leq C + 2 \vr + \delta$.
This contradicts our assumption that $T P_B$ is unbounded.
\end{proof}

\subsection{$C_0(K,X)$ spaces}\label{ss:auto C(K)}

If $X$ is a Banach lattice, and $K$ is a locally compact Hausdorff space,
let $C_0(K,X)$ denote the space of continuous functions $f:K\rightarrow X$,
having the property that, for any $\vr > 0$, there exists a compact set $\Omega$
so that $\|f(t)\| < \vr$ whenever $t \notin \Omega$.
We endow $C_0(K,X)$ with the norm $\|f\|=\sup_{t\in K}\|f(t)\|_X$, thus
turning it into a Banach lattice with the pointwise order.

\begin{theorem}\label{t:cont on CK}
Suppose $X$ is a K\"othe function space on a $\sigma$-finite measure space $(\Omega,\mu)$
with the Fatou property, and $K$ a locally compact Hausdorff space.
Then any $\vr$-BP linear map on $C_0(K,X)$ is automatically continuous.
\end{theorem}

Applying this theorem with $X = \RR$, we conclude that any $\vr$-BP linear map on $C_0(K)$
is automatically continuous.

For the proof we need a topological result (cf.~\cite{Per}).

\begin{lemma}\label{l:disj seq}
Suppose $(s_n)_{n \in \NN}$ are distinct points in a locally compact
Hausdorff space $K$.
Then there exist a family of disjoint open sets $(U_k)_{k \in \NN}$
so that $s_{n_k} \in U_k$ for any $k$ ($n_1 < n_2 < \ldots$).
\end{lemma}

\begin{proof}
We construct the sequence $(n_k)$, and the open sets $U_k$, recursively.

Note first that for any sequence of distinct points $(t_i)_{i \in \NN}$ in a Hausdorff
space there is at most one natural number $m$ so that any neighborhood
of $t_m$ contains all but finitely many members of the sequence $(t_i)$.
Indeed, if there exist two numbers, say $m$ and $\ell$, with this property,
then $t_m$ and $t_\ell$ cannot be separated, which cannot happen in a Hausdorff topology.

Consequently, if $(t_i)_{i \in \NN}$ is a sequence of distinct points in a locally compact
 Hausdorff space, then for any $i \in I$ (where $I$ is either $\NN$ or
$\NN \backslash \{m\}$, for the $m$ corresponding to the sequence $(t_i)_{i \in \NN}$)
there exists an open neighborhood $V_i$ of $t_i$ so that
$\{j \in I : t_j \notin \overline{V_i}\}$ is infinite.

Let $S_0 = \NN$. Pick $n_0 \in S_0$ in such a way that $s_{n_0}$ has an open neighborhood
$U_0$ so that $S_1:=\{n \in S_0 : s_n \notin \overline{U_0}\}$ is infinite.

Now suppose we have already selected $n_0 < \ldots < n_{k-1}$, and disjoint open sets
$U_0, \ldots, U_{k-1}$, so that $s_{n_j} \in U_j$ for $0 \leq j \leq k-1$, and
$$
S_k = \{n \in \NN : s_n \notin \cup_{j=1}^{k-1} \overline{U_j}\}
$$
is infinite.
Find $n_k \in S_k$ with an open neighborhood $V_k$ so that
$$
S_{k+1}:=\{n \in S_k : s_n \notin \overline{V_k} \}
$$
is infinite. Note that the same property holds for $U_k=V_k\backslash\bigcup_{j=0}^{k-1} \overline{U_j}$.

Proceed further in the same manner to obtain a sequence with the desired properties.
\end{proof}

We now proceed to prove Theorem \ref{t:cont on CK}.
For the rest of this subsection, $K$ is locally compact Hausdorff,
unless specified otherwise.

Suppose $T : C_0(K,X) \to C_0(K,X)$ is an $\vr$-BP linear map. For $t \in K$, let
$$
\lambda_t = \|\delta_t T\| = \sup \big\{ \big\| [Tf](t) \big\|_X : \|f\| \leq 1 \big\}
 \in [0,\infty] .
$$
We want to show that $\sup_{t\in K} \lambda_t < \infty$.

\begin{lemma}\label{l:vectorval}
Suppose $X$ is a Banach lattice, and $T : C_0(K,X) \to C_0(K,X)$ is an $\vr$-BP linear map.
If $f\in C_0(K,X)$ vanishes on an open set $V\subset K$, then $\|[Tf](t)\|_X\leq\vr\|f\|$
for any $t \in V$.
\end{lemma}

\begin{proof}
By Urysohn's Lemma, there is a continuous function $h:K\rightarrow[0,1]$
such that $h(t)=1$ and $h(s)=0$ for $s\in V^c$. Let $\phi=|Tf(t)|\cdot h\in C_0(K,X)$.
We have that $\phi\perp f$ and since $T$ is $\vr$-BP, it follows that
$$
\|[Tf](t)\|_X=\||[Tf](t)|\wedge \phi(t)\|_X\leq \||Tf|\wedge \phi\|\leq\vr\|f\|.
\qedhere
$$
\end{proof}

\begin{lemma}\label{l:approx}
For any $t \in K$, any open neighborhood $U$ with $t\in U$, and any $\sigma > 0$,
there exists $f \in \ball(C_0(K,X))$ so that $f$ vanishes outside of $U$, and
$\|[Tf](t)\|_X > \lambda_t - \vr - \sigma$.
\end{lemma}

\begin{proof}
Pick $g \in \ball(C_0(K,X))$ so that $\|[Tg](t)\|_X > \lambda_t - \sigma$.
Find an open set $V$ so that $\overline{V}$ is compact,
and $t \in V \subset \overline{V} \subset U$.
Urysohn's Lemma allows us to find a function $h$ so that $0 \leq h \leq 1$,
$h|_{\overline V} = 1$, and $h|_{U^c} = 0$. Let $f = hg$, and $f^\prime = (\one - h)g$.
Since $f'|_V=0$, Lemma \ref{l:vectorval} gives
$\big\| [Tf^\prime](t) \big\|_X \leq \vr\|f^\prime\|\leq \vr$.
By the triangle inequality,
$$
\big\| [Tf](t) \big\|_X \geq \big\| [Tg](t) \big\|_X - \big\| [Tf^\prime](t) \big\|_X >
 \lambda_t - \sigma - \vr .
\qedhere
$$
\end{proof}

\begin{lemma}\label{l:sequence}
If $(t_k)$ is a sequence of distinct points in $K$, then $\limsup_k \lambda_{t_k} < \infty$.
\end{lemma}

\begin{proof}
Suppose otherwise. Passing to a subsequence, we can assume that
$\lambda_{t_k} > 4^k + \vr$ for any $k$. Applying Lemma \ref{l:disj seq}, and
passing to a further subsequence if necessary, we can assume that
there exist disjoint open sets $U_k$ such that $t_k \in U_k$ for every $k$.
By Lemma \ref{l:approx}, we can find $f_k \in \ball(C_0(K,X))$,
vanishing outside of $U_k$, so that $\|[Tf_k](t_k)\|_X > 4^k$.

Now let $f = \sum_k 2^{-k} f_k$. Clearly $f \in C_0(K,X)$ (with $\|f\| \leq 2$),
and for every $n$,
$$
\sum_{k \neq n} 2^{-k} f_k|_{U_n}=0.
$$
Hence, by Lemma \ref{l:vectorval} we have
$$
\big\| [Tf](t_n) \big\|_X \geq 2^{-n} \big\| [Tf_n](t_n) \big\|_X -
 \vr \big\| \sum_{k \neq n} 2^{-k} f_k \big\| > 2^n - 2 \vr ,
$$
which contradicts the fact that $Tf\in C_0(K,X)$.
\end{proof}

\begin{lemma}\label{l:converge}
If $t_n \to t$, then $\lambda_t \leq \limsup \lambda_{t_n}$.
\end{lemma}

\begin{proof}
Suppose, for the sake of contradiction, that there exists a sequence $(t_n)$ converging
to $t$, and $\lambda_t > c > \sup_n \lambda_{t_n}$. Pick $f \in \ball(C_0(K,X))$
so that $\|[Tf](t)\|_X > c$. On the other hand,
$$
\|[Tf](t)\|_X = \lim_n \|[Tf](t_n)\|_X \leq c,
$$
a contradiction.
\end{proof}

\begin{theorem}\label{t:noisolated}
Let $K$ be a locally compact Hausdorff space without isolated points,
and $X$ a Banach lattice. If $T:C_0(K,X)\rightarrow C_0(K,X)$ is a linear $\vr$-BP mapping,
then $T$ is bounded.
\end{theorem}

\begin{proof}
As noted above, we need to show that $\sup_{t \in K} \lambda_t < \infty$.
Since $K$ has no isolated points, for every $t\in K$ there is a sequence $(t_k)$
of distinct points in $K$, such that $t_k\rightarrow t$. By Lemmas \ref{l:sequence} and
\ref{l:converge}, it follows that $\lambda_t$ is finite for every $t\in K$.

Suppose $\sup_{t \in K} \lambda_t = \infty$, then it would be possible to find
a sequence of distinct points $(t_n)$, so that $\lambda_{t_n}$ increases without a bound,
which is impossible by Lemma \ref{l:sequence}.
\end{proof}

\begin{proof}[Proof of Theorem \ref{t:cont on CK}]
We will prove first that for every $t \in K$, $\lambda_t$
is finite. Suppose, for the sake of contradiction, that $\lambda_t = \infty$
for some $t \in K$. By Lemmas \ref{l:sequence} and \ref{l:converge},
$t$ must be an isolated point in $K$. Hence, we can consider the function
$\chi_{\{t\}}\in C_0(K)$ as well as the operators $j_t:X\rightarrow C_0(K,X)$
and $\delta_t:C_0(K,X)\rightarrow X$ given by $j_t(x)=x\chi_{\{t\}}$
for $x\in X$, and $\delta_t(f)=f(t)$ for $f\in C_0(K,X)$ respectively.

Let $T_t=\delta_t T j_t$. It is clear that $T_t:X\rightarrow X$ is a linear mapping,
and we claim it is $\vr$-BP. Indeed, given $x,y\in X$ with $x\perp y$ and $y\geq0$,
we have that $\chi_{\{t\}}x\perp\chi_{\{t\}}y$ in $C_0(K,X)$, so as $T$ is $\vr$-BP
it follows that
$$
\||T(\chi_{\{t\}}x)|\wedge(\chi_{\{t\}}y)\|\leq \vr\|\chi_{\{t\}}x\|=\vr\|x\|.
$$
Therefore,
$$
\||T_t x|\wedge y\|=
\Big\|\delta_t\Big(|T(\chi_{\{t\}}x)|\wedge(\chi_{\{t\}}y)\Big)\Big\|\leq \vr\|x\|,
$$
so $T_t$ is $\vr$-BP as claimed. Proposition \ref{p:auto} yields that $T_t$ is bounded.

Since $t$ is isolated, any $f \in C_0(K,X)$ can be represented
as $f = f(t) \chi_{\{t\}} + f^\prime$, where $f^\prime$ vanishes at $t$
(equivalently, on a neighborhood of $t$), and moreover, $\|f^\prime\| \leq \|f\|$.
If $f \in \ball(C_0(K,X))$, then by Lemma \ref{l:vectorval}, we have
$$
\big\| [Tf](t) \big\|_X \leq \big\|  [T f(t)\chi_{\{t\}}](t) \big\|_X + \vr=\|T_t f\|+\vr .
$$
Taking the supremum over all $f$ as above, we obtain
$\lambda_t \leq \|T_t\| + \vr < \infty$, a contradiction.

Suppose, for the sake of contradiction, that $T$ is unbounded --
that is, $\sup_{t \in K} \lambda_t = \infty$.
By the above, there must exist a sequence $(t_k)_{k \in \NN} \subset K$
so that $\lim_k \lambda_{t_k} = \infty$. This, however, contradicts
Lemma \ref{l:sequence}.
\end{proof}

\section{Some notions related to $\vr$-band preservation}\label{s:center}

In this section, we consider some properties related to
(and perhaps strengthening) band preservation.

\begin{definition}
An operator on a Banach lattice $T: X\to X$ is $\vr$-approximable by BP maps
(in short $T\in ABP(\vr)$) when there is a BP operator $S$ such that
$\|T-S\|\leq\vr$.
\end{definition}

Clearly every $T\in ABP(\vr)$ is bounded, and $\vr$-BP.
In Section \ref{s:stability} we will study under which conditions every $\vr$-BP operator is in $ABP(\vr)$.

Recall (Theorem \ref{t:BP=>IP}) that $T \in B(X)$ is $\vr$-BP if and only if
for every $x \in \ball(X)$ and $\vr^\prime > \vr$ there exists
$\lambda = \lambda_x > 0$ such that $\| (|Tx| - \lambda|x|)_+\| < \vr'$.
However, in principle, we have no control over $\sup_{x \in \ball(X)} \lambda_x$.
Strengthening this properties, we introduce:

\begin{definition}
An operator $T:X\rightarrow X$ is in the $\vr$-center
(in short $T\in\vr-\mathcal{Z}(X)$) if there exists $\lambda>0$ such that for every
$x\in\ball(X)$, there is $z\in X$ with $\|z\|\leq\vr$ such that $$|Tx|\leq\lambda|x|+z.$$
\end{definition}

Note that $T\in 0-\mathcal{Z}(X)=\mathcal{Z}(X)$ if and only if $T$ is BP (\cite{AVK}). Moreover, if $T$ is BP, and $S$ is arbitrary, then $T + S \in \|S\|-\mathcal{Z}(X)$. In general, if $T \in \vr-\mathcal{Z}(X)$, then $T$ is $\vr$-BP. We do not know whether the converse implication holds in general. However, if $T\in ABP(\vr)$, then $T\in\vr-\mathcal Z(X)$. In Section \ref{s:stability} we will provide conditions for which every $\vr$-BP operator is in $ABP(4\vr)$, hence it also belongs to $4\vr-\mathcal Z(X)$.

Note that $T\in \vr-\mathcal Z(X)$ if and only if there is $\lambda \geq0$ such that for every $x\in\ball(X)$,
$$
\|(|Tx|-\lambda|x|)_+\|\leq\vr.
$$

For $T\in\vr-\mathcal Z(X)$, we define
$$
\rho_\vr(T)=\inf\{\lambda\geq0: \sup_{x \in \ball(X)} \|(|Tx|-\lambda|x|)_+\|\leq\vr\}.
$$

\begin{proposition}\label{p:duality center}
Let $X$ be a Banach lattice and $\vr\geq0$. Given $T \in B(X)$, we have
$T\in \vr-\mathcal Z(X)$ if and only if $T^*\in\vr-\mathcal Z(X^*)$.
Moreover, $\rho_\vr(T)=\rho_\vr(T^*)$.
\end{proposition}

\begin{proof}
Suppose $T\in \vr-\mathcal Z(X)$ and take $\lambda>\rho_\vr(T)$.
By the Riesz-Kantorovich formulas (cf. \cite[Theorem 1.18, and p. 58]{A-B}),
given $x\in\ball(X)_+$ and $x^*\in\ball(X^*)$ we have
\begin{align*}
\Big\langle \Big(|T^*x^*|-\lambda|x^*|\Big)_+,x\Big\rangle
&
=\sup_{0\leq y\leq x}\langle|T^*x^*|-\lambda|x^*|,y\rangle
\\
&
=\sup_{0\leq y\leq x}\Big(\sup_{|z|\leq y}\Big|\langle T^*x^*,z\rangle\Big|-\lambda\langle|x^*|,y\rangle\Big)
\\
&
=\sup_{0\leq y\leq x}\sup_{|z|\leq y}\Big(\Big|\langle x^*,Tz\rangle\Big|-\lambda\langle|x^*|,y\rangle\Big)
\\
&
\leq\sup_{0\leq y\leq x}\sup_{|z|\leq y}\Big(\langle |x^*|,|Tz|-\lambda|z|\rangle\Big)
\\
&
\leq\sup_{0\leq y\leq x}\sup_{|z|\leq y}\|x^*\|\Big\|\Big(|Tz|-\lambda|z|\Big)_+\Big\|\leq\vr.
\end{align*}
Therefore, $T^*\in\vr-\mathcal Z(X^*)$ and $\rho_\vr(T^*)\leq\rho_\vr(T)$.

Now, suppose $T^*\in \vr-\mathcal Z(X^*)$. Applying the above argument to $T^*$ we obtain that $T^{**}\in \vr-\mathcal Z(X^{**})$ with $\rho_\vr(T^{**})\leq\rho_\vr(T^*)$. Since $T^{**}|_X=T$, this implies that $T\in \vr-\mathcal Z(X)$ and
$$
\rho_\vr(T)\leq\rho_\vr(T^{**})\leq\rho_\vr(T^*)\leq\rho_\vr(T).
\qedhere
$$
\end{proof}

\begin{definition}
An operator on a Banach lattice $T: X\to X$ is locally $\vr$-approximable by BP maps (in short $T\in ABP_{loc}(\vr)$) provided for every $x\in X$, there is a BP operator $S_x$ such that
$$
\|Tx-S_x x\|\leq\vr\|x\|.
$$
\end{definition}

It is clear that every operator $T\in ABP_{loc}(\vr)$ is $\vr$-BP. Moreover, if the local approximants $S_x$ can be taken in such a way that $\sup_x\|S_x\|<\infty$, then $T\in\vr-\mathcal Z(X)$, with $\rho_\vr(T)\leq\sup_x\|S_x\|$. The following provides a converse:

\begin{theorem}\label{t:loc appr}
Suppose $E$ is a Banach lattice with a quasi-interior point, $\vr > 0$, and $T \in B(E)$.
\begin{enumerate}
 \item
$T$ is $\vr$-BP if and only if $T \in ABP_{loc}(\vr')$ for every $\vr^\prime > \vr$.
\item
$T \in \vr - {\mathcal{Z}}(E)$ with $\rho_\vr(T) < C$ if and only if for every
$x \in \ball(E)$ and every $\vr^\prime > \vr$ there exists a BP map $T_x \in B(E)$
so that $\|T_x\| < C$, and $\|Tx - T_x x\| < \vr^\prime$.
\end{enumerate}
\end{theorem}

Before the proof we need a decomposition result.

\begin{lemma}\label{l:decomp}
Suppose $x$, $y$, and $z$ are elements of a Banach lattice $E$, so that
$|y| \leq |x| + z$. Then there exists $u \in E$ so that $\|y-u\| \leq \|z\|$,
and $|u| \leq |x|$.
\end{lemma}

\begin{proof}[Sketch of a proof]
Without loss of generality, we may assume $z \geq 0$.
We have
$$
\big\| |y| - |y| \wedge |x| \big\| =
\big\| |y| \wedge (|x| + z) - |y| \wedge |x| \big\| =
\big\| |y| \wedge \big((|x| + z) - |x| \big) \big\| \leq \|z\| .
$$
It remains to show that there exists $u \in E$ so that $|u| = a := |y| \wedge |x|$
and $\|u - y\| = \| |u| - |y| \|$.
To this end, recall that the ideal $I_y$ generated by $|y|$ can be identified with $C(K)$, for some $K$ (with $|y|$ corresponding to $\one$).
Further, $y$ can be identified with
the function $y(t) = |y|(t) w(t)$, where $|w| = 1$. % whenever $r \neq 0$.
We can set $u(t) = a(t) w(t)$.
\end{proof}

\begin{proof}[Proof of Theorem \ref{t:loc appr}]
(1) Suppose first that, for any $x \in \ball(E)$, and any $\vr^\prime > \vr$, we can
find a BP map $T_x$ so that $\|Tx - T_xx\| < \vr^\prime$.
Then $|Tx| \leq |T_x x| + |Tx - T_x x| \leq \|T_x\| |x| + |Tx - T_x x|$.
If $y \geq 0$ is disjoint from $x$, then
$|Tx| \wedge y \leq \|T_x\| |x| \wedge y + |Tx - T_x x| \wedge y$
has norm not exceeding $\|Tx - T_xx\|$. From the definition, $T$ is $\vr$-BP.

Suppose, conversely, that $T$ is $\vr$-BP. By Theorem \ref{t:BP=>IP},
for any $\vr^\prime > \vr$, and any $x \in \ball(E)$, there exists $\lambda > 0$
so that $|Tx| \leq \lambda |x| + z$, where $\|z\| < \frac{\vr+\vr^\prime}{2}$.
By Lemma \ref{l:decomp}, there exists $y \in E$ with $|y| \leq \lambda |x|$,
and $\|y - Tx\| < \vr^\prime$. Since $E$ has a quasi-interior point, by \cite[Lemma 4.17]{AA}, there exists $T_x \in B(E)$ band preserving such that $|T_x z| \leq \lambda|z|$, for every $z\in E$ and $\|T_xx- y\|\leq\frac{\vr^\prime-\vr}{2}$. Hence,
$$
\|Tx-T_xx\|\leq\|Tx-y\|+\|T_xx-y\|\leq\vr^\prime.
$$

(2) is handled similarly.
\end{proof}

\begin{remark}
For a Dedekind complete Banach lattice $X$ and $\vr\geq0$, a similar argument using  \cite[Theorem 2.49]{A-B} yields that if $T\in\vr-\mathcal Z(X)$, then $T\in ABP_{loc}(2\vr)$ with local approximants satisfying $\sup_x\|S_x\|\leq2\rho_{\vr}(T)$. Also if $T$ is an $\vr$-BP operator, then it is locally $2\vr$-approximable by BP maps.
\end{remark}

\bigskip

The following diagram illustrates the relation among the different notions introduced here,
for bounded operators.
The non-trivial implications are labeled with the reference of the corresponding result
where they are proved. Note the values of $\vr$ may differ from one to another, and some
of the implications are proved only for some classes of Banach lattices.
$$
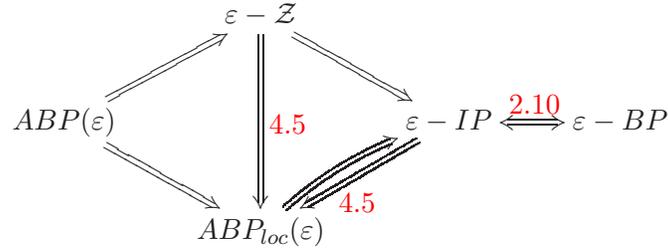
\xymatrix{&\vr-\mathcal Z\ar@{=>}[rd]\ar@{=>}[dd]^{\ref{t:loc appr}}&&\\
ABP(\vr)\ar@{=>}[ur]\ar@{=>}[dr]&&\vr-IP\ar@{=>}@//[ld]^{\ref{t:loc appr}}\ar@{<=>}[r]^{\ref{t:BP=>IP}}&\vr-BP\\
&ABP_{loc}(\vr)\ar@{=>}@/^/[ur]&&}
$$
For unbounded operators the picture is different: we do not know whether
$\vr$-BP implies $\vr$-IP.

In Section \ref{s:counter}, we show that some of the arrows on the diagram
cannot be reversed: for every $\vr > 0$ there exists a contraction in
$\vr-{\mathcal{Z}}$ (and in $ABP_{loc}(\vr)$), but not in $ABP(\delta)$
for $\delta < 1/2$.

\section{Stability of almost band preservers}\label{s:stability}

We will show now that, under some mild hypothesis,
an almost band-preserving operator is close to a band-preserving one.

\begin{theorem}\label{t:BP ocont}
If $E$ is an order continuous Banach lattice, and $T \in B(E)$ is $\vr$-$\BP$, then there exists a band-preserving
$R \in B(E)$ so that $\|R\| \leq \|T\|$, and $\|T-R\| \leq 4\vr$.
\end{theorem}

For positive $\vr$-BP operators a similar result holds under weaker assumptions on $X$.

\begin{proposition}\label{p:ocont}
Suppose $X$ is a Dedekind complete
Banach lattice having a Fatou norm with constant $\fatou$.
Then for any positive $\vr$-BP operator $T\in B(X)$ there exists $0\leq S\leq T$
such that $S$ is band-preserving and $\|T-S\|\leq 4\fatou\vr$.
\end{proposition}

Note that order continuous Banach lattices, and dual Banach lattices, are
Dedekind complete, and have Fatou norm with constant $1$ (see e.g.
\cite[Theorem 2.4.2 and Proposition 2.4.19]{M-N}).

For the proof of Proposition \ref{p:ocont}, we need to introduce an order
in the family of finite sets of band projections.
These can be considered as an abstract version of partitions of unity:

\begin{definition}
Given a Banach lattice $E$, let ${\mathcal{P}}$ be the family of finite sets of
band projections $P = (P_1, \ldots, P_n)$ so that $P_i P_j = 0$ whenever $i \neq j$,
and $\sum_{k=1}^n P_k = I_E$. We say that $P = (P_1, \ldots, P_n) \prec Q = (Q_1, \ldots, Q_m)$ if for $1 \leq i \leq n$
there exists a set $S_i \subset \{1, \ldots, m\}$ so that $\sum_{j \in S_i} Q_j = P_i$.
\end{definition}

Note that the order $\prec$ makes ${\mathcal{P}}$ into a net: for $P = (P_1, \ldots, P_n), Q = (Q_1, \ldots, Q_m) \in {\mathcal{P}}$, we can define the family $R$ consisting of band projections $R_{ij} = P_i Q_j$, which satisfies $P,Q\prec R$.

As a preliminary step toward Theorem \ref{t:BP ocont}, we establish:

\begin{lemma}\label{l:BP dual}
Let $E$ be an order continuous Banach lattice, and $T \in B(E^*)$ is $\vr$-$\BP$,
then there exists a band-preserving
$U \in B(E^*)$ so that $\|U\| \leq \|T\|$, and $\|T-U\| \leq 4\vr$.
\end{lemma}

\begin{proof}
For $P = (P_1, \ldots, P_n) \in {\mathcal{P}}$ on $E^*$, define $T_P = \sum_{k=1}^n P_k T P_k$. Since $T$ is $\vr$-BP, by Proposition \ref{p:equiv}, for every $S\subset \{1,\ldots,n\}$ we have that
$$
\Big\|\Big(\sum_{i \in S} P_i\Big) T \Big(\sum_{i \in S^c} P_i\Big)\Big\|\leq\vr.
$$
Note that
$$
\sum_{S\subset\{1,\ldots,n\}} \sum_{i\in S}\sum_{j\in S^c} P_iTP_j=\sum_{\substack{i,j=1\\i\neq j}}^n\sum_{\substack{S\subset\{1,\ldots,n\}\\i\in S,j\notin S}}P_iTP_j=2^{n-2}\sum_{\substack{i,j=1\\i\neq j}}^nP_iTP_j.
$$
Thus,
$$
T - T_P = \sum_{i \neq j} P_i T P_j = 4 \ave_{S \subset \{1, \ldots, n\}}
 \Big(\sum_{i \in S} P_i\Big) T \Big(\sum_{i \in S^c} P_i\Big) ,
$$
hence $\|T - T_P\| \leq 4 \vr$ for every $P\in\mathcal P$.

Recall that we have $B(E^*) = (E^* {\hat{\otimes}} E)^*$
via the trace duality: $\langle A, e^* \otimes e \rangle = \langle Ae^*, e \rangle$,
for $e \in E$, $e^* \in E^*$, and $A \in B(E^*)$ (see e.g. \cite[Section 1.1.3]{DFS}).
Thus, the operators $T_P \in B(E^*,E^*)$
have a subnet convergent weak$^*$ to $U \in B(E^*,E^*)$,
with $\|T - U\| \leq 4 \vr$.

Finally, we show that, for any band projection $R \in B(E^*)$, we have $R U R^\perp = 0$
(as $E^*$ is $\sigma$-Dedekind complete, the band-preserving property of $U$ will follow).
For ``large enough'' $P\in\mathcal P$ (that is, when $(R,R^\perp)\prec P$), we have
$R T_P R^\perp = 0$. From the definition of $U$, $T_P \to U$ in the point-weak$^*$
topology. By \cite[Corollary 2.4.7]{M-N}, $R$ and $R^\perp$ are weak$^*$ to weak$^*$
continuous, hence $T_P \to U$ in the point-weak$^*$ topology as well.
Thus, $R T_P R^\perp = 0$.
\end{proof}

\begin{proof}[Proof of Theorem \ref{t:BP ocont}]
Since $T$ is $\vr$-BP and $E$ is order continuous,
by Proposition \ref{p:duality}, we have that $T^*$ is also $\vr$-BP.
By Lemma \ref{l:BP dual}, there exists a band-preserving $U \in B(E^*)$
so that $\|U\| \leq \|T^*\|$, and $\|T^*-U\| \leq 4\vr$.

Now, since $U$ is band preserving we have that $-\|U\| I \leq U\leq\|U\| I$,
which means that for $x\in E$,
$$
|U^*x|\leq\|U\||x|.
$$
Since $E$ is order continuous, it is an ideal in $E^{**}$, and the above inequality yields that $U^*(E)\subset E$. In particular, $R=U^*|_E:E\rightarrow E$ is well defined and satisfies $R^*=U$. By Proposition \ref{p:duality}, it follows that $R$ is band preserving in $E$. Moreover, we have
$$
\|T-R\|=\|T^*-R^*\|=\|T^*-U\|\leq 4\vr.
\qedhere $$ \end{proof}

The following easy lemma may well be known, but we haven't seen it stated explicitly.

\begin{lemma}\label{l:inf}
Suppose ${\mathcal{A}}$ and ${\mathcal{B}}$ are bounded below sets in a
Dedekind complete Banach lattice $X$. Then
$$
\bigwedge_{a \in {\mathcal{A}}} a + \bigwedge_{b \in {\mathcal{B}}} b =
\bigwedge_{a \in {\mathcal{A}}, b \in {\mathcal{B}}} (a+b) .
$$
\end{lemma}

\begin{proof}
Without loss of generality we can assume
$\bigwedge_{a \in {\mathcal{A}}} a = 0 = \bigwedge_{b \in {\mathcal{B}}} b$,
then clearly $\bigwedge_{a \in {\mathcal{A}}, b \in {\mathcal{B}}} (a+b) \geq 0$.
To prove the converse, note that, for any $b_0 \in {\mathcal{B}}$,
$$
\bigwedge_{a \in {\mathcal{A}}, b \in {\mathcal{B}}} (a+b) \leq
\bigwedge_{a \in {\mathcal{A}}} (a+b_0) = b_0 .
$$
Complete the proof by taking the infimum over $b_0 \in {\mathcal{B}}$.
\end{proof}

\begin{proof}[Proof of Proposition \ref{p:ocont}]
For $P = (P_1, \ldots, P_n) \in {\mathcal{P}}$, we define $$T_P = \sum_{k=1}^n P_k T P_k.$$
As in the proof of Theorem \ref{t:BP ocont}, since $T$ is $\vr$-BP, we have
$\|T - T_P\| \leq 4 \vr$.

Since $T$ is positive, for every $x\in X_+$, the net $(T_Px)_{P\in\mathcal{P}}$ is decreasing.
Indeed, let $P = (P_1, \ldots, P_n) \prec Q = (Q_1, \ldots, Q_m)$. Thus, for
$1 \leq i \leq n$ there exists a set $S_i \subset \{1, \ldots, m\}$ so that
$\sum_{j \in S_i} Q_j = P_i$. In particular, $P_i\geq Q_j$ for every $j\in S_i$, and we get
$$
T_Px=\sum_{i=1}^n P_iTP_ix =
 \sum_{i=1}^n\sum_{j\in S_i} Q_jTP_ix \geq
 \sum_{i=1}^n\sum_{j\in S_i} Q_jTQ_jx = T_Qx.
$$

Since $X$ is Dedekind complete, % $(T_Px)_{P\in\mathcal{P}}$ converge in norm to
$\bigwedge_{P\in\mathcal{P}}T_Px$ exists. For each $x\in X_+$, let
$Sx = \bigwedge_{P\in\mathcal{P}}T_Px$.
Then $S$ defines an additive positively homogeneous function on $X_+$.
The homogeneity is easy to verify: for any $\lambda \geq 0$ and $x \in X_+$,
we have
$$
S(\lambda x) = \bigwedge_{P\in\mathcal P}(\lambda T_P x) =
\lambda \bigwedge_{P\in\mathcal P} (T_P x) = \lambda S x .
$$
The positive additivity follows directly from Lemma \ref{l:inf}.

Clearly, $0\leq S\leq T$. Also, for any $x \in X_+$,
$(T-S)x = \bigvee_P (T - T_P)x$, hence, by the Fatou Property,
$$
\|(T-S)x\| = \sup_{P\in\mathcal P} \|(T - T_P)x\| \leq 4\fatou \vr \|x\|.
$$
As $T-S$ is a positive operator, $\|T-S\| \leq 4\fatou \vr$.

It remains to see that $S$ is band preserving.
Given a band projection $R$ and $x \in X_+$, we have
$R T_P R^\perp x = 0$ for $P$ ``large enough'' (that is, when $(R,R^\perp)\prec P$).
Therefore, for $x \in X_+$,
$$
0 \leq R S R^\perp x \leq \bigwedge_P R T_P R^\perp x = 0,
$$
which implies $RSR^\perp=0$.
\end{proof}

It should be noted that the hypotheses of Theorem \ref{t:BP ocont}  and
Propositions \ref{p:ocont} are not always necessary: In Theorem \ref{t:C(K)} we will see
that on $C(K)$ spaces every $\vr$-BP operator is close to a BP one.

Suppose now that $E$ is a Banach lattice. Under what conditions on $E$
does there exist $c > 0$ so that, for every $\vr > 0$, $E$ can be equipped
with a new lattice norm $\triple{ \cdot }$ so that there exists
a $\vr$-BP operator on $(E, \triple{ \cdot })$ with the property that
$\triple{T-S} \geq c$ for every BP operator $S$? By Theorem \ref{t:BP ocont}, this cannot happen when
$E$ is order continuous (order continuity passes to renormings).
A partial positive answer is given below.

\begin{proposition}\label{p:atom lim}
Suppose $E$ is a Banach lattice so that its dual has an atom $f$ with
the property that $f^\perp = \{g \in E^* : f \perp g \}$ is not weak$^*$ closed.
Then, for every $\vr \in (0,1)$, $E$ can be equipped with an equivalent lattice norm
$\triple{ \cdot } = \triple{ \cdot }_\vr$ so that there exists a positive rank one
contraction $T \in \vr-{\mathcal{Z}}((E, \triple{ \cdot}))$ so that $\triple{T - S} \geq c$
whenever $S$ is a BP map ($c > 0$ is a constant depending on $E$).
\end{proposition}

As noted above, $E$ is order continuous if and only if any
band in $E^*$ is weak$^*$ closed if and only if any band projection on $E^*$ is weak$^*$
continuous. Of course, there may be no rank one band projections on $E^*$ at all.
We do not know whether Proposition \ref{p:atom lim} holds for general
non-order continuous lattices.

Proposition \ref{p:atom lim} is applicable, for instance, when $E = C(K)$,
where $K$ is an infinite compact Hausdorff space (cf. Remark \ref{r:renorm C(K)}).

\begin{proof}
Without loss of generality, we can assume that $f$ is positive and has norm one.
Note that $f^\perp$ is a $1$-codimensional sublattice of $E^*$.
Indeed, let $P$ be the (one-dimensional) band projection corresponding to $f$,
then $f^\perp$ is the range of $P^\perp = I - P$.

As $f^\perp$ is not weak$^*$ closed, by the Banach-Dieudonn\'e Theorem,
$f$ is a cluster point of $\{g \in f^\perp : \|g\| \leq C\}$, for some $C$.

Fix $\vr \in(0,1)$, and equip $E$ with the new lattice norm
$$
\triple{x} = \max \big\{ \|x\| , \vr^{-1} \langle f, |x| \rangle \big\} .
$$
Note that $\| \cdot \| \leq \triple{ \cdot } \leq \vr^{-1} \| \cdot \|$.
It is easy to check that the dual norm on $E^*$ is given by
$$
\triple{g} = \inf_{\alpha \in \RR} \big( \vr |\alpha| + \|g - \alpha f\| \big) =
 \inf_{g = \alpha f + h} \big( \vr |\alpha| + \|h\| \big) .
$$

Fix $\delta > 0$, and find a positive norm one $e \in E$ so that
$\langle f, e \rangle > 1 - \delta$. Define the rank one positive map
$T : E \to E : x \mapsto \langle f, x \rangle e$. It is easy to check
that $T$ acts contractively on $(E, \triple{ \cdot })$. Indeed, if
$\triple{x} \leq 1$, then
$$
\big| \langle f , x \rangle \big| \leq \langle f , |x| \rangle \leq \vr ,
$$
hence
$$
\triple{Tx} = | \langle f , x \rangle | \max \big\{\|e\|, | \vr^{-1} \langle f , e \rangle | \big\} \leq 1 .
$$

Further, $T^* g = \langle g, e \rangle f$. We show that
$T^* \in \vr-{\mathcal{Z}}(E^*, \triple{ \cdot })$
(then, by Proposition \ref{p:duality center},
$T \in \vr-{\mathcal{Z}}(E, \triple{ \cdot })$).
Pick $g \in E^*$ with $\triple{g} < 1$. Write $g = \alpha f + h$,
with $\vr |\alpha| + \|h\| < 1$. We need to show that
$|T^*g| \leq |g| + u$, with $\triple{u} \leq \vr$.
As $T$ is positive, we can and do restrict our attention to $g \geq 0$,
and to the decompositions with $\alpha \geq 0$ and $h \geq 0$.

Then
$T^* g = (\alpha \langle f,e \rangle + \langle h,e \rangle) f$, hence
$T^*g \leq (\alpha + \|h\|) f \leq g + \|h\| f$.
However, $\triple{f} \leq \vr$, and we are done.

Next show that, if $S$ is a BP map on $E$, then $\triple{T-S} \geq (1-\delta)/(C+1)$.
Recall $P$ is the band projection associated with $f$,
and then $P^\perp$ is the band projection onto $f^\perp$. We have
$$
\triple{T^* - S^*} \geq \triple{ P^\perp(T^* - S^*)|_{f^\perp} } = \triple{S^*|_{f^\perp}}
$$
(we use the fact that $S^*$ maps $f^\perp$ into itself).
Thus, for any $g \in f^\perp$, we have $|S^* g| \leq c |g|$, where $c = \triple{T-S}$.

Using the band-preserving property of $S^*$ once more, we observe
that $S^* f = \lambda f$, for some scalar $\lambda$. We claim that
$|\lambda| \leq Cc$. Indeed, we know that $f$ is a weak$^*$ cluster point of
$\{g \in f^\perp : \|g\| \leq C\}$. As $S^*$ is weak$^*$ to weak$^*$ continuous,
$S^*f$ is a weak$^*$ cluster point of $ \{S^*g : g \in f^\perp, \|g\| \leq C\}$,
which, in turn, lies inside $\{g \in f^\perp : \|g\| \leq cC\}$.

On the other hand, note that
$$
\triple{T^* f} = | \langle f, e \rangle | \triple{f} > (1-\delta) \triple{f}.
$$
The triangle inequality implies
$$
c = \triple{T^* - S^*} \geq \frac{\triple{ (T^* - S^*)f }}{\triple{f}} = \frac{\triple{T^*f} - \triple{S^*f}}{\triple{f}} >
(1 - \delta - Cc).
$$
Thus,
$$
\triple{T-S} \geq(1-\delta)/(C+1).
\qedhere
$$
\end{proof}

\section{$\vr$-BP operators on $C(K)$ spaces}\label{s:C(K)}

In this section, we turn our attention to operators on $C(K)$ spaces.
Let us start by presenting a criterion for a linear map on $C(K)$ to be $\vr$-BP.
Before the proof, recall that for $f\in C(K)$, we define its support as
$$\supp(f)=\overline{\{t\in K:f(t)\neq0\}}.$$

\begin{lemma}\label{l:BP C(K)}
Suppose $K$ is a compact Hausdorff space.
Then, for a linear map $T : C(K) \to C(K)$, the following statements are equivalent.
\begin{enumerate}
\item
$T$ is $\vr$-BP.
\item
If $x \in C(K)$ and $t \in K$ satisfy $x(t) = 0$, then $|[Tx](t)| \leq \vr \|x\|$.
\end{enumerate}
\end{lemma}

\begin{proof}
$(2) \Rightarrow (1)$: Let $x\perp y$ in $C(K)$. We have that
\begin{eqnarray*}
\||Tx|\wedge|y|\|&=&\sup\{|[Tx](t)|\wedge|y(t)|:t\in K\}\\
&=&\sup\{|[Tx](t)|\wedge|y(t)|:t\in K, \,y(t)\neq 0\}\\
&\leq&\vr\|x\|,
\end{eqnarray*}
where the last inequality follows from the fact that if $y(t)\neq0$, then $x(t)=0$, together with the hypothesis.

$(1) \Rightarrow (2)$: Suppose first $t \notin \supp(x)$. By Urysohn's Lemma, there
exists $y \in C(K)$ so that $0 \leq y \leq |[Tx](t)|$, $y|_{\supp(x)} = 0$,
and $y(t) = |[Tx](t)|$. Then $x \perp y$, and
$$
|[Tx](t)|\leq\||Tx| \wedge y\|\leq\vr\|x\|.
$$

Now suppose $t \in \partial \supp(x)$. For $\delta > 0$, let
$$
x_\delta = (x_+ - \delta \one)_+ - (x_- - \delta \one)_+.
$$
Note that $x_+ - \delta \one \leq (x_+ - \delta \one)_+ \leq x_+$, hence
$\|x_+ - (x_+ - \delta \one)_+\| \leq \delta$. Similarly,
$\|x_- - (x_- - \delta \one)_+\| \leq \delta$. Thus, by the triangle inequality,
$\|x - x_\delta\| \leq 2\delta $.

Moreover, we claim that $t \notin \supp(x_\delta)$. Indeed, let us consider
the open set $U=\{s\in K:|x(s)|<\delta\}$. Clearly, $t\in U$, and for every
$s\in U$ $x_\delta(s)=0$. Thus, $U\cap\supp(x_\delta)=\emptyset$. By the preceding
paragraph, $|[Tx_\delta](t)| \leq \vr \|x_\delta\|$. As $T$ is continuous
(Theorem \ref{t:cont on CK}), we are done.
\end{proof}

\begin{theorem}\label{t:C(K)}
Suppose $K$ is a compact Hausdorff space.
If $T \in B(C(K))$ is $\vr$-BP, then there exists a BP operator $S \in B(C(K))$ so that
$\|T-S\| \leq 2\vr$. If $T$ is positive, then $S$ can be selected to be positive as well.
\end{theorem}

\begin{proof}
Let $\phi = T \one$, and show that the multiplication operator $S$ defined via
$S f = \phi f$ has the desired properties. To this end, for $t \in K$,
set $\mu_t = T^* \delta_t \in M(K) = C(K)^*$. Let $c_t = \mu_t(\{t\})$, and
$\nu_t = \mu_t - c_t \delta_t$. Clearly $\nu_t(\{t\}) = 0$.
We claim that $\|\nu_t\| \leq \vr$ -- that is, for any $f \in \ball(C(K))$,
$|\langle \nu_t, f \rangle| \leq \vr$.

Suppose first $f$ satisfies an additional condition $f(t) = 0$. Then, since $T$ is $\vr$-BP, by Lemma \ref{l:BP C(K)}, we get
$$
\vr \geq |Tf(t)| = | \langle \delta_t, T f \rangle | =
|\langle T^* \delta_t , f\rangle| = |\langle \nu_t, f \rangle| .
$$
For a generic $f \in \ball(C(K))$, fix $\sigma > 0$. By the regularity of $\nu_t$,
there exists an open neighborhood $U \ni t$ so that $|\nu_t|(U) < \sigma$.
Use Urysohn's Lemma to find $h \in C(K)$ so that $0 \leq h \leq 1$, $h(t) = 1$,
and $h|_{K \backslash U} = 0$. By the above, $|\langle \nu_t, fh \rangle| \leq \sigma$,
and $|\langle \nu_t, f(1-h) \rangle| \leq \vr$, which yields
$|\langle \nu_t, f \rangle| \leq \vr + \sigma$. As $\sigma$ can be arbitrarily
small, $|\langle \nu_t, f \rangle| \leq \vr$. % we are done.

Next note that
$$
\phi(t) = \langle T^* \delta_t, \one \rangle = c_t + \langle \nu_t, \one \rangle ,
$$
hence $|\phi(t) - c_t| \leq \vr$.  Finally, for $f \in \ball(C(K))$,
$$
[Tf](t) = \langle \delta_t, Tf \rangle = \langle c_t \delta_t + \nu_t, f \rangle =
c_t f(t) + \langle \nu_t, f \rangle ,
$$
hence
$$
\big| [Tf](t) - \phi(t) f(t) \big| \leq |c_t - \phi(t)| + \|\nu_t\| \leq 2 \vr .
$$
As this holds for any $t$, we are done.
\end{proof}

\begin{remark}
As a consequence, on $(C(K),\|\cdot\|_\infty)$, every $\vr$-BP operator belongs to
$2\vr-{\mathcal{Z}}(C(K))$).
\end{remark}

Recall that an operator between Banach lattices $T:X\rightarrow Y$ is $\vr$-disjointness preserving (in short, $\vr$-DP) if for $x\perp y$ in $X$ with $\|x\|,\|y\|\leq 1$ we have $\| |Tx| \wedge |Ty| \| \leq \vr$ \cite{OikTra}.

\begin{proposition}\label{p:BP=>DP C(K)}
If $K$ is a compact Hausdorff space, and $T \in B(C(K))$ is $\vr$-BP,
then $T$ is $\vr$-DP. Moreover, if
$x, y \in C(K)$ are disjoint, then $\| (Tx) (Ty) \| \leq \vr \|T\| \|x\| \|y\|$.
\end{proposition}

\begin{proof}
Consider disjoint $x, y \in C(K)$. By Lemma \ref{l:BP C(K)},
$\| |Tx| \wedge |Ty| \| \leq \vr \max\{\|x\|, \|y\| \}$. Thus, $T$ is $\vr$-DP.

Moreover, for $t \notin \supp(x)$,
$$
|[Tx](t)| |[Ty](t)| \leq \vr \|x\| \cdot \|T\| \|y\| =
\vr \|T\| \|x\| \|y\| ,
$$
and the same inequality holds for $t \notin \supp(y)$.
\end{proof}

\begin{remark}\label{r:renorm C(K)}
Suppose $K$ is a Hausdorff compact.
Recall that a space $C(K)$ is order continuous if and only if $K$ is a finite set.
If $K$ is infinite, then Proposition \ref{p:atom lim} gives a renorming of $C(K)$
for which the conclusion of Theorem \ref{t:C(K)} no longer holds. We outline
the construction from Proposition \ref{p:atom lim} for $C(K)$ spaces,
as it may be instructive. In fact, for $\vr > 0$ we equip $C(K)$
with an equivalent norm $\triple{ \cdot}$, and construct a positive contraction
$T \in \vr-{\mathcal{Z}}(C(K), \triple{ \cdot })$ so that $\triple{T-S} \geq 1/2$
for any BP operator $S$.

Since $K$ is infinite, it has an accumulation point $k$. Consider the norm
$$
\triple{x} = \max\{ \|x\|_\infty , \vr^{-1} |x(k)| \}
$$
($\| \cdot \|_\infty$ stands for the canonical $\sup$ norm on $C(K)$).
Clearly $\triple{ \cdot }$ is a lattice norm on $C(K)$, and
$\| x \|_\infty \leq \triple{x} \leq \vr^{-1} \| x \|_\infty$ for $x\in C(K)$.
We denote $(C(K), \triple{ \cdot})$ by $E$.

Consider the rank one operator $T: E \to E: x \mapsto x(k) \one$.
 Clearly $T \geq 0$, and for $x\in E$ we have
$$
\triple{Tx} = \triple{x(k) \one} = \max \{|x(k)|,\vr^{-1}|x(k)|\} \leq \triple{x} .
$$
Note also that $T\one=\one$, hence $\triple{T}=1$.
Also, $T\in\vr-{\mathcal{Z}}(E)$. Indeed, for $x\in\ball(E)$, set
$y = (|Tx| - |x|)_+$. Then $y(k) = 0$, while for $t \in K \backslash \{k\}$,
$|y(t)| \leq \vr$, hence $\triple{y} \leq \vr$. % and we have $|Tx|\leq|x|+y$.

Now suppose, for the sake of contradiction, that a BP map $S:E\rightarrow E$
satisfies $\triple{T - S} < 1/2$. Note that $S$ is a multiplication operator:
there exists $\phi \in C(K)$ so that $Sx = \phi x$.
We claim that $|\phi(k)|< 1/2$. Indeed, take a net
$(k_\alpha)_\alpha\subset K\backslash\{k\}$ such that $k_\alpha\rightarrow k$.
By Urysohn's Lemma, for every $\alpha$ there is $x_\alpha\in C(K)$ such that
$0 \leq x_\alpha \leq x_\alpha(k_\alpha) = 1$, and $x_\alpha(k) = 0$.
Then $Tx_\alpha = 0$, hence $|\phi(k_\alpha)| \leq \|(S-T)x_\alpha\| \leq \|T-S\| < 1/2$.
By continuity, $|\phi(k)| \leq \triple{T-S} < 1/2$ as well.
On the other hand, since $\triple{\one}=\vr^{-1}$, we have
$$
\triple{T-S} \geq \triple{(T - S)\vr\one} \geq \vr^{-1} |[(T-S)\vr\one](k)| =
 |1 - \phi(k)| > \frac12 ,
$$
which is the desired contradiction.
\end{remark}

\section{A counterexample}\label{s:counter}

\begin{proposition}\label{p:better counter}
There exists a Banach lattice $E$ so that, for every $\vr > 0$, there
exists an $\vr$-BP contraction $T \in B(E)$ (actually, $T\in\vr-\mathcal Z(E)$) so that $\|T - S\| \geq 1/2$
whenever $S \in B(E)$ is BP.
\end{proposition}

\begin{lemma}\label{l:C_0(K)}
Suppose $K$ is a compact Hausdorff space. For $t_0 \in K$,
$C(K;t_0) = \{x \in C(K): x(t_0) = 0\}$ is a closed sublattice of $C(K)$.
Then $T \in B(C(K;t_0))$ is BP if and only if there exists
a uniformly bounded continuous function $\phi$ on $K \backslash \{t_0\}$
so that $Tx = \phi x$ for any $x$.
\end{lemma}

\begin{proof}
Clearly the operator given by $T(x)=\phi x$ is BP. Conversely, suppose $T$ is BP. Then $T$ is
automatically bounded. Note that $C(K;t_0)^*$ is the quotient space of $C(K)^*$ by the
set of linear functionals annihilating $\{t_0\}$. That is, we can identify
$C(K;t_0)^*$ with the space of regular Radon measures $\mu$ on $K$ so that
$\mu(\{t_0\}) = 0$.

Consider $T^* \in B(C(K;t_0)^*)$. We claim that
there exists $\phi : K \backslash \{t_0\} \to {\mathbb{K}}$ so that
$T^* \delta_t = \phi(t) \delta_t$ for any $t \in K \backslash \{t_0\}$. Indeed, fix $t$, and set $\mu_t = T^* \delta_t$.
Suppose, for the sake of contradiction, that $|\mu_t|(K \backslash \{t\}) > 0$.
Then there exists an open set $U \supset \{t, t_0\}$ so that $|\mu_t|(K \backslash U) > 0$.
Then we can find $x \in C(K)$, vanishing on $U$, so that $\langle \mu_t , x \rangle > 0$.

Also, find $y \in C(K;t_0)$, vanishing outside of $U$, so that $y(t) = 1$.
Then $x \perp y$. However,
$$
\big[ Tx \big](t) = \langle \delta_t , Tx \rangle = \langle T^*\delta_t , x \rangle =
\langle \mu_t , x \rangle > 0 ,
$$
hence $[|Tx| \wedge |y|](t)\neq 0$, contradicting our assumption that $T$ is BP.

Next show that $\phi$ is continuous and uniformly bounded. For $x \in C(K;t_0)$
and $t \neq t_0$, we have
$$
[Tx](t) = \langle \delta_t, Tx \rangle = \langle T^* \delta_t, x \rangle =
\langle \phi(t) \delta_t, x \rangle = \phi(t) x(t) .
$$
This shows the continuity of $\phi$ away from $t_0$. If $\phi$ is not uniformly
bounded, then there exists a sequence $(t_k)$, convergent to $t_0$, so that
$|\phi(t_k)| > 4^k$ for any $k$. By Tietze Extension Theorem, we can find
$x \in C(K;t_0)$ so that $x(t_k) = 2^{-k}$. Then $Tx$ is unbounded, leading to a contradiction.
\end{proof}

\begin{proof}[Proof of Proposition \ref{p:better counter}]
The Banach lattice $E$ consists of all continuous functions $x$ on $[0,1]$,
satisfying $\underset{n \rightarrow\infty}{\lim}  2^n |x(2^{-n})| = 0$
(consequently $x(0) = 0$). Set
$$
\|x\| = \max \big\{ \|x\|_\infty, \sup_{n \in \NN} 2^n |x(2^{-n})| \big\} .
$$
For $n \geq 2$ let $x_n$ be a continuous function such that $0 \leq x_n \leq 1$,
$x_n(t) = 0$ for $t \leq 2^{-n-1}$ or $t\geq 2^{1-n}$, and $x_n(2^{-n}) = 1$.
Define $T_n x = x(2^{-n}) x_n$.

Clearly, $T_n$ is a contraction. We next show that, for every $x \in \ball(E)$,
$\| (|T_n x| - |x|)_+ \| \leq 2^{-n}$ (hence, $T_n \in 2^{-n}-\mathcal Z(E)$, and in particular $T_n$ is $2^{-n}$-BP).

If $x(2^{-n}) = 0$, then $T_n x = 0$. Otherwise $x(2^{-n}) = [T_n x](2^{-n})$,
and $[T_n x](2^{-m}) = 0$ for $m \neq n$. For $t \notin \{2^{-m} : m \in \NN\}$,
$$
(|T_n x| - |x|)_+(t) \leq |[T_n x](t)| \leq |x(2^{-n})| |x_n(t)| \leq 2^{-n}.
$$
Thus, we get
$$
\| (|T_n x| - |x|)_+ \| = \max \big\{ \| (|T_n x| - |x|)_+ \|_\infty,
  \sup_m 2^m (|T_n x| - |x|)_+(2^{-m})\big\} \leq 2^{-n} .
$$

Now suppose $S: E \to E$ is band-preserving (hence bounded).
We show that that there exists a uniformly bounded continuous function
$\phi : (0,1] \to {\mathbb{K}}$ so that $Sx = \phi x$ for any $x$.

Indeed, for any $n \in \NN$, denote by $E_n$ the sublattice of $E$
consisting of functions vanishing on $[0,2^{-n}]$. Note that $S$ takes $E_n$
into itself. Clearly $E_n$ is lattice isomorphic to $C([2^{-n},1],2^{-n})$,
hence, by Lemma \ref{l:C_0(K)}, there exists a uniformly bounded continuous function
$\phi_n : (2^{-n},1] \to {\mathbb{K}}$ so that $Sx = \phi_n x$
for any $x \in E_n$. Clearly $\phi_m|_{[2^{-n},1]} = \phi_n$ whenever $m > n$.
So there exists a function $\phi$, continuous on $(0,1]$, so that $Sx = \phi x$
for any $x \in E_\infty$, where $E_\infty=\bigcup_{n} E_n$ is the set of all elements of $E$
vanishing on a neighborhood of $0$.

Now set $C = \|S\|$, and show that $\sup_{t \in (0,1]} |\phi(t)| \leq C$.
Indeed, otherwise we can find $t \in (0,1] \backslash \{2^{-k} : k \in \NN\}$
so that $|\phi(t)| > C$. Find $m \in \NN$ so that $2^{-m} < t < 2^{1-m}$,
and consider $x \in C[0,1]$ so that $0 \leq x \leq 1 = x(t)$,
and $x = 0$ outside of $(2^{-m}, 2^{1-m})$. Then $\|x\| = 1$ and $\|Sx\| > C$,
a contradiction.

It is easy to see that $E_\infty$ is dense in $E$, hence by continuity,
$Sx = \phi x$ for any $x \in E$.

Now suppose, for the sake of contradiction, that there exists a BP map $S \in B(E)$
so that $\|T_n - S\| = c < 1/2$. We have shown that $S$ is implemented by
multiplication by a function $\phi$, continuous on $(0,1]$ and uniformly bounded.
That is, for any $x \in \ball(E)$, we have $\|T_nx - \phi x\| \leq c$.

Show first that, for $t \notin \{2^{-k} : k \in \NN\}$, $|\phi(t)| \leq c$.
To this end, find $n \in \NN$ so that $2^{-n-1} < t < 2^{-n}$.
Pick $x \in E$ so that $0 \leq x \leq 1 = x(t)$, and $x = 0$
outside of $(2^{-n-1}, 2^{-n})$. Then $\|x\| = 1$, $T_nx = 0$, and
$c \geq \|Sx\| \geq |\phi(t)|$. By continuity, $|\phi| \leq c$ everywhere.

Now consider $x \in E$ so that $x(2^{-n}) = 2^{-n}$, $0 \leq x \leq 1$, and
$x = 0$ outside of $(2^{-n-1}, 2^{1-n})$. Then $\|x\| = 1$, and
\begin{align*}
\|T_n x - Sx\|
&
\geq 2^n \big|[T_n x](2^{-n}) - \phi(2^{-n}) x(2^{-n})\big|
\\
&
=
2^n \big| 2^{-n} - 2^{-n} \phi(2^{-n}) \big| \geq
1 - |\phi(2^{-n})| \geq 1-c > \frac12 ,
\end{align*}
a contradiction.
\end{proof}

\begin{remark}\label{r:AM space}
The lattice $E$ from the proof of Proposition \ref{p:better counter} is an AM-space.
In fact,
$j : E \to C[0,1] \oplus_\infty c_0 : f \mapsto f \oplus (2^k f(2^{-k}))_{k \in \NN}$
is a lattice isometry.
Consequently, $E^*$ is an AL-space.
As $T_n^* \in \vr-{\mathcal{Z}}(E^*)$ for any $n$,
Theorem \ref{t:BP ocont} shows there exists a BP map $R_n \in B(E^*)$ so that
$\|T_n^* - R_n\| \leq 2^{2-n}$
However, such an $R_n$ cannot be an adjoint operator, for $n>3$.
\end{remark}

\begin{remark}\label{r:what can we do}
Arguing as in Theorem \ref{t:cont on CK} one can show that every $\vr$-BP linear map on the lattice
$E$ given in Proposition \ref{p:better counter} is automatically continuous.
\end{remark}


\begin{thebibliography}{22}


\bibitem{AVK} Y. A. Abramovich, A. I. Veksler, and A. V. Koldunov,
On operators preserving disjointness.
\emph{Soviet Math. Dokl.} 20 (1979), no. 5, 1089--1093

\bibitem{AA} Y. A. Abramovich and C. D. Aliprantis.
An invitation to operator theory.
Graduate Studies in Mathematics, 50. American Mathematical Society, Providence, RI, 2002.


\bibitem{A-B} C. D. Aliprantis and O. Burkinshaw.
{\rm Positive operators}.
Springer, Dordrecht, 2006.

\bibitem{DFS} J. Diestel, J. Fourie, and J. Swart.
{\em The metric theory of tensor products. Grothendieck's r\'esum\'e revisited}.
American Mathematical Society, Providence, RI, 2008.

\bibitem{Hui} C. B. Huijsmans,
Disjointness preserving operators on Banach lattices.
Operator theory in function spaces and Banach lattices, 173--189,
Oper. Theory Adv. Appl., 75, Birkh\"auser, Basel, 1995.

\bibitem{HdP} C. B. Huijsmans and B. de Pagter,
Ideal theory in f-algebras.
\emph{Trans. Amer. Math. Soc.} 269 (1982), no. 1, 225--245.

\bibitem{Kak} S. Kakutani,
Concrete representation of abstract $L$-spaces and the mean ergodic theorem.
\emph{Ann. of Math.} 42. 523-537 (1941).

\bibitem{LT2} J. Lindenstrauss and L. Tzafriri.
\newblock {\em Classical Banach spaces II},
\newblock Springer-Verlag, Berlin, 1979.

\bibitem{M-N} P. Meyer-Nieberg.
{\em Banach lattices}. Springer-Verlag, Berlin, 1991.

\bibitem{Nakano} H. Nakano,
Teilweise geordnete Algebra.
\emph{Jap. J. Math.} 17, (1941) 425--511.

\bibitem{OikTra} T. Oikhberg and P. Tradacete,
Almost disjointness preservers. Canad. J. Math. (to appear).

\bibitem{Per} A. Peralta,
Orthogonal forms and orthogonality preservers revisited,  Linear Multilinear Algebra
(to appear).


\bibitem{Zaa} A. C. Zaanen, Examples of orthomorphisms,
J. Approximation Theory 13 (1975), 192--204.



\end{thebibliography}
\end{document}